\documentclass[reqno, 11pt]{amsart}

\usepackage[text={155mm,235mm},centering]{geometry}
\input diagxy
\input xy

\usepackage[active]{srcltx} 

\newtheorem{thm}{Theorem}[section]
\newtheorem{cor}[thm]{Corollary}
\newtheorem{lem}[thm]{Lemma}
\newtheorem{prop}[thm]{Proposition}
\newtheorem{quest}[thm]{Question}

\theoremstyle{definition}

\theoremstyle{property}

\theoremstyle{remark}
\newtheorem{rem}[thm]{Remark}

\numberwithin{equation}{section}
\usepackage[mathscr]{eucal}
\usepackage[all]{xy}
\usepackage{mathrsfs}
\usepackage{xypic}
\usepackage{amsfonts}
\usepackage{amsmath}
\usepackage{amsthm,bm}
\usepackage{amssymb,tikz-cd}
\usepackage{latexsym}
\usepackage{tabularx}
\usepackage{graphicx}
\usepackage{pict2e}
\usepackage[pagebackref]{hyperref}
\usepackage{tikz}
\usepackage{textcomp}
\usepackage[scr=rsfs]{mathalfa}
\definecolor{ceruleanblue}{rgb}{0.16, 0.32, 0.75}
\hypersetup{colorlinks=true,allcolors=ceruleanblue}
\allowdisplaybreaks

\begin{document}

\title[Mayer-Vietoris systems and their applications]{Mayer-Vietoris systems and their applications}

\author{Lingxu Meng}
\address{Department of Mathematics, North University of China, Taiyuan, Shanxi 030051,  P. R. China}
\email{menglingxu@nuc.edu.cn}%

\subjclass[2010]{Primary 53C56; Secondary 14F25, 32C35}
\keywords{Mayer-Vietoris system; Leray-Hirsch theorem; blow-up formula; local system; locally free sheaf; Dolbeault cohomology}


\begin{abstract}
  We introduce notions such as  cdp presheaf,  cds precosheaf, Mayer-Vietoris system,  and investigate their properties. As applications, we study  cohomologies with values in local systems on smooth manifolds and Dolbeault cohomologies with values in locally free sheaves on complex manifolds, where the compactness is not necessary for both cases. In particular, we write out explicit blow-up formulas for these cohomologies. Moreover, we compare the blow-up formula given by Rao, S., Yang, S. and Yang, X.-D. with ours, and then deduce that their formula is still an isomorphism in the noncompact case.
\end{abstract}

\maketitle

\tableofcontents
\section{Introduction}
The Mayer-Vietoris sequences ae classical results in algebraic topology. They exist in various homology and cohomology theories satisfying the Eilenberg-Steenrod axioms (\cite{ES}). Moreover, it holds in topological K-theory (\cite{K}).  In \cite{M1, M2, M3}, we studied the Morse-Novikov cohomology, Dolbeault cohomology and gave their explicit formulas of blow-ups. There, we repeatedly used a technique that consider the local cases first, and then, extend them to global cases via Mayer-Vietoris sequences. On the smooth manifolds, we systematically develop this method  by the language of presheaves and precosheaves as follows.
\begin{thm}\label{1.1}
Let $X$ be a connected smooth manifold and $\mathcal{M}^*$, $\mathcal{N}^*$  M-V systems of cdp presheaves or cds precosheaves on $X$. Assume that $F^*:\mathcal{M}^*\rightarrow \mathcal{N}^*$ is a  M-V morphism  satisfying the following hypothesis:

\emph{(*)} There exists a basis $\mathfrak{U}$ of topology of $X$, such that, $F^*(U_1\cap...\cap U_l)$ is an isomorphism  for any finite $U_1,..., U_l\in \mathfrak{U}$.\\
Then $F^*$ is an isomorphism.
\end{thm}
As applications of this theorem, we generalize the main results in \cite{CY, RYY, RYY2}. During our preparation of the present work, Rao, S., Yang, S. and Yang, X.-D. (\cite{RYY2}) gave a blow-up formula for bundle-valued Dolbeault cohomology on compact complex manifolds. With the similar way in \cite{RYY2}, Chen, Y. and Yang, S. (\cite{CY}) gave a blow-up formula for cohomology with values in local systems on compact complex manifolds. By Theorem \ref{1.1}, we will give other formulas in a different way and remove the compactness. Moreover, we will prove their formulas are still isomorphic on the noncompact bases.

Let $\pi:\widetilde{X}\rightarrow X$ be the blow-up of a connected complex manifold $X$ along a connected complex submaifold $Y$. We know $\pi|_E:E=\pi^{-1}(Y)\rightarrow Y$ is the projective bundle $\mathbb{P}(N_{Y/X})$ associated to the normal bundle $N_{Y/X}$ over $Y$. Assume that $i_Y:Y\rightarrow X$, $i_E:E\rightarrow \widetilde{X}$ are inclusions and $r=\textrm{codim}_{\mathbb{C}}Y$. Set $t=\frac{i}{2\pi}\Theta(\mathcal{O}_{E}(-1))\in \mathcal{A}^{1,1}(E)$, where $\mathcal{O}_{E}(-1)$ is the universal line bundle  on $E={\mathbb{P}(N_{Y/X})}$ and  $\Theta(\mathcal{O}_{E}(-1))$ is the Chern curvature of a hermitian metric on $\mathcal{O}_{E}(-1)$. Clearly, $\textrm{d}t=0$ and $\bar{\partial}t=0$.
\begin{thm}[Theorem \ref{blow-up}, \ref{important} and Section 6.3]\label{1.2}
Assume that $X$, $Y$, $\pi$, $\widetilde{X}$, $E$, $t$, $i_Y$, $i_E$ are defined as above.  Then
\begin{displaymath}
\pi^*+\sum_{i=0}^{r-2}(i_E)_*\circ (h^i\cup)\circ (\pi|_E)^*
\end{displaymath}
gives isomorphisms
\begin{displaymath}
H^k(X,\mathcal{V})\oplus \bigoplus_{i=0}^{r-2}H^{k-2-2i}(Y,i_Y^{-1}\mathcal{V})\tilde{\rightarrow} H^k(\widetilde{X}, \pi^{-1}\mathcal{V}),
\end{displaymath}
\begin{displaymath}
H_{c}^k(X,\mathcal{V})\oplus \bigoplus_{i=0}^{r-2}H_{c}^{k-2-2i}(Y,i_Y^{-1}\mathcal{V})\tilde{\rightarrow} H_{c}^k(\widetilde{X},\pi^{-1}\mathcal{V}),
\end{displaymath}
for any $k$, where $h=[t]\in H^2(E,\mathbb{R})$ or $H^2(E,\mathbb{C})$  and $\mathcal{V}$ is a local system of $\mathbb{R}$ or $\mathbb{C}$-modules of finite rank on $X$, and an isomorphism
\begin{displaymath}
H^{p,q}(X,\mathcal{E})\oplus \bigoplus_{i=0}^{r-2}H^{p-1-i,q-1-i}(Y,i_Y^*\mathcal{E})\tilde{\rightarrow} H^{p,q}(\widetilde{X},\pi^*\mathcal{E}),
\end{displaymath}
for any $p$, $q$, where $h=[t]\in H^{1,1}(E)$ and $\mathcal{E}$ is a locally free sheaf of $\mathcal{O}_X$-modules of finite rank on $X$.

In the inverse direction,
\begin{displaymath}
\phi^{\mathcal{V}}:H^k(\widetilde{X}, \pi^{-1}\mathcal{V})\tilde{\rightarrow} H^k(X,\mathcal{V})\oplus \bigoplus_{i=0}^{r-2}H^{k-2-2i}(Y,i_Y^{-1}\mathcal{V}),
\end{displaymath}
\begin{displaymath}
\phi_c^{\mathcal{V}}:H_{c}^k(\widetilde{X},\pi^{-1}\mathcal{V})\tilde{\rightarrow} H_{c}^k(X,\mathcal{V})\oplus \bigoplus_{i=0}^{r-2}H_{c}^{k-2-2i}(Y,i_Y^{-1}\mathcal{V}),
\end{displaymath}
and
\begin{displaymath}
\phi:H^{p,q}(\widetilde{X},\pi^*\mathcal{E})\rightarrow H^{p,q}(X,\mathcal{E})\oplus \bigoplus_{i=0}^{r-2}H^{p-1-i,q-1-i}(Y,i_Y^*\mathcal{E})
\end{displaymath}
are isomorphisms, where $\phi^{\mathcal{V}}$, $\phi_c^{\mathcal{V}}$ and $\phi$ are defined by Section 6.2 \emph{(\ref{def})} and Section 6.3.
\end{thm}

In Section 2, we introduce the notions such as cdp presheaf, cds precosheaf, M-V system and prove Theorem \ref{1.1}. In Section 3,  we define operators on forms and currents with values in local systems and locally free sheaves. Moreover, we  generalize R. O. Wells' main results in \cite{W}. In Section 4,  examples of M-V systems and morphisms are given. In Section 5, we show some applications of Theorem \ref{1.1}, for instance, Poincar\'{e} dulaity theorem, K\"{u}nneth formula, Leray-Hirsch theorems on cohomology with values in local systems and Dolbeault cohomology with values in locally free sheaves. In Section 6, we verify Theorem \ref{1.2} and put forward several questions on these blow-up formulas.

\subsection*{Acknowledgements}
I would like to express my gratitude to School of Mathematics and Statistics, Wuhan University for the warm hospitality during my visit. I sincerely thank  Prof. Jiangwei Xue for pointing out to me the meaningless notions in a  previous version of the present article. I would like to thank Prof. Sheng Rao and Dr. Xiang-Dong Yang for  sending me their articles  \cite{RYY2}, \cite{YY} and their helpful discussions.

\section{Mayer-Vietoris systems}
In this section, $R$ denotes a commutative ring with unit. If presheaves, precosheaves and morphisms are mentioned, they will be assumed to be  them of $R$-modules implicitly.
\subsection{Definitions}
Let $X$ be a topological space and $\mathcal{M}$ a presheaf (resp. precosheaf) on $X$. For open sets $V\subseteq U$, denote by $\rho_{U,V}$ (resp. $i_{V,U}$) the restriction $\mathcal{M}(U)\rightarrow\mathcal{M}(V)$ (resp. the extension $\mathcal{M}(V)\rightarrow\mathcal{M}(U)$). We call $\mathcal{M}$ a \emph{cdp presheaf} (resp. \emph{cds precosheaf}), if it satisfies the \emph{countable direct product} (resp. \emph{sum}) \emph{condition}, i.e., for any collection $\{U_n|n\in\mathbb{N}^+\}$ of disjoint open subsets of $X$,
\begin{displaymath}
(\rho_{\bigcup_{n=1}^\infty U_n,U_n})_{n\in\mathbb{Z}}:\mathcal{M}(\bigcup_{n=1}^\infty U_n)\rightarrow\prod_{n=1}^\infty\mathcal{M}(U_n)
\end{displaymath}
(resp.
\begin{displaymath}
\sum_{n=1}^\infty i_{U_n,\bigcup_{n=1}^\infty U_n}:\bigoplus_{n=1}^\infty\mathcal{M}(U_n)\rightarrow\mathcal{M}(\bigcup_{n=1}^\infty U_n))
\end{displaymath}
is an isomorphism.  Clearly, any sheaf (resp. cosheaf) is a cdp presheaf (resp. cds precosheaf).

A collection  $\mathcal{M}^*=\{(\mathcal{M}^p, \mbox{ }\delta^p)|p\in \mathbb{Z}\}$ is called \emph{a M-V system of presheaves} (resp. \emph{precosheaves}) on $X$, if

$(i)$ For any $p\in\mathbb{Z}$, $\mathcal{M}^p$ is a presheaf (resp. precosheaf) on $X$;

$(ii)$ For  any open subsets $U$ and $V$,  $\delta_{U,V}^p:\mathcal{M}^p(U\cap V)\rightarrow\mathcal{M}^{p+1}(U\cup V)$ (resp. $\delta_{U,V}^p:\mathcal{M}^{p}(U\cup V)\rightarrow\mathcal{M}^{p+1}(U\cap V)$) are morphisms for all $p\in\mathbb{Z}$, and satisfy that
\begin{displaymath}
\tiny{
\xymatrix{
\cdots\mathcal{M}^{p-1}(U\cap V)\ar[r]^{\quad\delta_{U,V}^{p-1}}& \mathcal{M}^p(U\cup V)\ar[r]^{P_{U,V}^p\quad}&\mathcal{M}^p(U)\oplus \mathcal{M}^p(V)\ar[r]^{\quad Q_{U,V}^p}&\mathcal{M}^{p}(U\cap V)\ar[r]^{\delta_{U,V}^{p}\quad}&\mathcal{M}^{p+1}(U\cup V)\cdots}}
\end{displaymath}
(resp.
\begin{displaymath}
\tiny{
\xymatrix{
\cdots\mathcal{M}^{p-1}(U\cup V)\ar[r]^{\quad\delta_{U,V}^{p-1}}& \mathcal{M}^p(U\cap V)\ar[r]^{P_{U,V}^p\quad}&\mathcal{M}^p(U)\oplus \mathcal{M}^p(V)\ar[r]^{\quad Q_{U,V}^p}&\mathcal{M}^{p}(U\cup V)\ar[r]^{\delta_{U,V}^{p}\quad}&\mathcal{M}^{p+1}(U\cap V)\cdots}})
\end{displaymath}
is an exact sequence, where
\begin{displaymath}
P_{U,V}^p(\alpha)=(\rho^p_{U\cup V,U}(\alpha),\mbox{ }\rho^p_{U\cup V,V}(\alpha))\mbox{ and }Q_{U,V}^p(\beta,\gamma)=\rho^p_{U,U\cap V}(\beta)-\rho^p_{V,U\cap V}(\gamma)
\end{displaymath}
(resp.
\begin{displaymath}
P_{U,V}^p(\alpha)=(i^p_{U\cap V,U}(\alpha),\mbox{ }i^p_{U\cap V,V}(\alpha))\mbox{ and }Q_{U,V}^p(\beta,\gamma)=i^p_{U,U\cup V}(\beta)-i^p_{V,U\cup V}(\gamma)).
\end{displaymath}
If $\mathcal{M}^p$ are sheaves (resp. cosheaves) for all $p$, then $\mathcal{M}^*=\{(\mathcal{M}^p, \mbox{ }\delta^p)|p\in \mathbb{Z}\}$ is a M-V system of presheaves (resp. precosheaves) on $X$, if and only if, $\delta_{U,V}^p=0$ and $Q_{U,V}^p$ (resp. $P_{U,V}^p$) are surjective (resp. injective) for all $p\in\mathbb{Z}$ and open sets $U$, $V$.

Assume that $\mathcal{M}^*$ and $\mathcal{N}^*$ are M-V systems of presheaves (resp. precosheaves) on $X$. We say $F^*:\mathcal{M}^*\rightarrow \mathcal{N}^*$ is \emph{a morphism of M-V systems}, or briefly, \emph{a M-V morphism}, if $F^*=\{F^p:\mathcal{M}^p\rightarrow \mathcal{N}^p|p\in \mathbb{Z}\}$ is a collection of morphisms of presheaves (resp. precosheaves) satisfying that the diagram
\begin{displaymath}
\tiny{
\xymatrix{
  \cdots\mathcal{M}^{p-1}(U\cap V)\ar[d]^{F^{p-1}(U\cap V)}\ar[r]^{\quad\delta_{U,V}^{p-1}}& \mathcal{M}^p(U\cup V)\ar[d]^{F^{p}(U\cup V)}\ar[r]^{P_{U,V}^p\quad}&\mathcal{M}^p(U)\oplus \mathcal{M}^p(V)\ar[d]^{(F^p(U),F^p(V))}\ar[r]^{\quad Q_{U,V}^p}&\mathcal{M}^{p}(U\cap V)\ar[d]^{F^{p}(U\cap V)}\ar[r]^{\delta_{U,V}^{p}\quad}&\mathcal{M}^{p+1}(U\cup V)\ar[d]^{F^{p+1}(U\cup V)}\cdots\\
\cdots\mathcal{N}^{p-1}(U\cap V)\ar[r]^{\quad\delta_{U,V}^{p-1}}& \mathcal{N}^p(U\cup V)\ar[r]^{P_{U,V}^p\quad}&\mathcal{N}^p(U)\oplus \mathcal{N}^p(V)\ar[r]^{\quad Q_{U,V}^p}&\mathcal{N}^{p}(U\cap V)\ar[r]^{\delta_{U,V}^{p}\quad}&\mathcal{N}^{p+1}(U\cup V)\cdots }}
\end{displaymath}
(resp.
\begin{displaymath}
\tiny{
\xymatrix{
  \cdots\mathcal{M}^{p-1}(U\cup V)\ar[d]^{F^{p-1}(U\cup V)}\ar[r]^{\quad\delta_{U,V}^{p-1}}& \mathcal{M}^p(U\cap V)\ar[d]^{F^{p}(U\cap V)}\ar[r]^{P_{U,V}^p\quad}&\mathcal{M}^p(U)\oplus \mathcal{M}^p(V)\ar[d]^{(F^p(U),F^p(V))}\ar[r]^{\quad Q_{U,V}^p}&\mathcal{M}^{p}(U\cup V)\ar[d]^{F^{p}(U\cup V)}\ar[r]^{\delta_{U,V}^{p}\quad}&\mathcal{M}^{p+1}(U\cap V)\ar[d]^{F^{p+1}(U\cap V)}\cdots\\
 \cdots\mathcal{N}^{p-1}(U\cup V)\ar[r]^{\quad\delta_{U,V}^{p-1}}& \mathcal{N}^p(U\cap V)\ar[r]^{P_{U,V}^p\quad}&\mathcal{N}^p(U)\oplus \mathcal{N}^p(V)\ar[r]^{\quad Q_{U,V}^p}&\mathcal{N}^{p}(U\cup V)\ar[r]^{\delta_{U,V}^{p}\quad}&\mathcal{N}^{p+1}(U\cap V)\cdots)}}
\end{displaymath}
is commutative.

For  $n\in\mathbb{Z}$ and a M-V system $\mathcal{M}^*=\{(\mathcal{M}^p,\mbox{ } \delta^p)|p\in \mathbb{Z}\}$,  $\mathcal{M}^*[n]=\{(\mathcal{M}^p[n], \mbox{ }\delta^p[n])|p\in \mathbb{Z}\}$ is also a M-V system, where $\mathcal{M}^p[n]=\mathcal{M}^{p+n}$ and $\delta^p[n]=\delta^{p+n}$.

Suppose that $F^*:\mathcal{M}^*\rightarrow \mathcal{N}^*$ is a M-V morphism  on $X$.

$(1)$ For an open set $U\subseteq X$, $F^*(U)$ is called  \emph{an isomorphism}, if $F^p(U):\mathcal{M}^p(U)\rightarrow \mathcal{N}^p(U)$ are isomorphisms for all $p$.

$(2)$ $F^*$ is called \emph{an isomorphism}, if $F^*(U)$ is  an isomorphism for any open set $U\subseteq X$.

We easily get the following proposition.
\begin{prop}\label{elem}
Suppose that $f:X\rightarrow Y$ is a continuous map of topological spaces.

$(1)$ For a  cdp presheaf \emph{(}resp. cds precosheaf\emph{)} $\mathcal{M}$  on $X$, the direct image $f_*\mathcal{M}$ is a cdp presheaf \emph{(}resp. cds precosheaf\emph{)} on $Y$.

$(2)$ For a M-V system $\mathcal{M}^*=\{(\mathcal{M}^p,\mbox{ } \delta^p)|p\in \mathbb{Z}\}$ of presheaves \emph{(}resp. precosheaves\emph{)} on $X$,
the direct image $f_*\mathcal{M}^*:=\{(f_*\mathcal{M}^p,\mbox{ } f_*\delta^p)|p\in \mathbb{Z}\}$ is a M-V system of presheaves \emph{(}resp. precosheaves\emph{)} on $Y$.

$(3)$ For a M-V morphism $F^*:\mathcal{M}^*\rightarrow \mathcal{N}^*$ on $X$,  $f_*F^*:f_*\mathcal{M}^*\rightarrow f_*\mathcal{N}^*$ is a M-V morphism on $Y$, where $f_*F^*=\{f_*F^p:f_*\mathcal{M}^p\rightarrow f_*\mathcal{N}^p|p\in\mathbb{Z}\}$.
\end{prop}

\subsection{Elementary Properties}
The following properties can be easily checked using the elementary homological algebra in the category of modules. Suppose that $X$ is a topological space.

For cdp presheaves and cds precosheaves on $X$, we have

$(1)$ $(i)$ For a cds precosheaf $\mathcal{M}$  and  a $R$-module $N$, then $\mathcal{M}\otimes_RN$ is a  cds precosheaf.

$(ii)$ Assume that $R$ is a Noether ring. For a cdp presheaf $\mathcal{M}$ and a finite generalized flat $R$-module $N$, $\mathcal{M}\otimes_RN$ is a cdp presheaf.

$(iii)$ For  a cdp presheaf (resp.  cds precosheaf)  $\mathcal{M}$ and a $R$-module $N$ , $\mathcal{H}om_{\underline{R}_X}(N,\mathcal{M})$ (resp. $\mathcal{H}om_{\underline{R}_X}(\mathcal{M},N)$) is a cdp presheaf, where $N$ is viewed as a constant presheaf (resp. precosheaf).

$(iv)$ For a cds precosheaf $\mathcal{M}$ and a finite generated $R$-module $N$, $\mathcal{H}om_{\underline{R}_X}(N,\mathcal{M})$ is a cds precosheaf.

$(2)$ If $\mathcal{M}_{\alpha}$ are  cdp presheaves (resp. cds precosheaves) for all $\alpha\in \Lambda$, then $\prod_{\alpha\in\Lambda}\mathcal{M}_\alpha$ (resp. $\bigoplus_{\alpha\in\Lambda}\mathcal{M}_\alpha$) is a cdp presheaf (resp. cds precosheaf).

$(3)$ For a  morphism $F:\mathcal{M}\rightarrow \mathcal{N}$  of cdp  presheaves (resp. cds precosheaves), $\textrm{ker}F$, $\textrm{Im}F$, $\textrm{coker}F$ are cdp presheaves (resp. cds precosheaves).

For M-V systems on $X$, we have

$(1')$ Suppose that $\mathcal{M}^*:=\{(\mathcal{M}^p, \mbox{ }\delta^p)|p\in \mathbb{Z}\}$ is a M-V system of presheaves (resp. precosheaves).

$(i)$ For a flat $R$-module $N$,
$\mathcal{M}^*\otimes_RN=\{(\mathcal{M}^p\otimes_RN,\mbox{ } \delta^p\otimes\textrm{id}_N)|p\in \mathbb{Z}\}$ is   a M-V system of presheaves (resp. precosheaves).

$(ii)$ For a projective $R$-module $N$, $\mathcal{H}om_{\underline{R}_X}(N,\mathcal{M}^*)$ is a M-V system of  presheaves (resp. precosheaves).

$(iii)$ For an injective $R$-module  $N$, $\mathcal{H}om_{\underline{R}_X}(\mathcal{M}^*,N)$ is a M-V system of  precosheaves (resp. presheaves). In particular, if $R$ is a divisible ring, $\mathcal{H}om_{\underline{R}_X}(\mathcal{M}^*,N)$ is a M-V system for any $R$-module $N$.

$(2')$ If $\mathcal{M}_\alpha^*=\{(\mathcal{M}_\alpha^p, \mbox{ }\delta_\alpha^p)|p\in \mathbb{Z}\}$ are M-V systems of presheaves (resp. precosheaves) for all $\alpha\in \Lambda$, then
\begin{displaymath}
\bigoplus_{\alpha\in\Lambda}\mathcal{M}_\alpha^*:=\{(\bigoplus_{\alpha\in\Lambda} \mathcal{M}_\alpha^p,\mbox{ } \bigoplus_{\alpha\in\Lambda}\delta_\alpha^p)|p\in \mathbb{Z}\}
\end{displaymath}
and
\begin{displaymath}
\prod_{\alpha\in\Lambda}\mathcal{M}_\alpha^*:=\{(\prod_{\alpha\in\Lambda} \mathcal{M}_\alpha^p, \mbox{ }\prod_{\alpha\in\Lambda}\delta_\alpha^p)|p\in \mathbb{Z}\}
\end{displaymath}
are both M-V systems of presheaves  (resp. precosheaves).

$(3')$ Let $F^*:\mathcal{L}^*\rightarrow \mathcal{M}^*$ and $G^*:\mathcal{M}^*\rightarrow \mathcal{N}^*$ be morphism of systems. Suppose that
\begin{displaymath}
\xymatrix{
 0\ar[r] & \mathcal{L}^*    \ar[r]^{F^*} & \mathcal{M}^*\ar[r]^{G^*}&\mathcal{N}^*\ar[r] &0 }
\end{displaymath}
is an exact sequence of systems of presheaves (resp. precosheaves), i.e., for every $p\in\mathbb{Z}$, $0\rightarrow L^p\rightarrow M^p\rightarrow N^p\rightarrow0$ is exact. Any two among $\mathcal{L}^*$, $\mathcal{M}^*$, $\mathcal{N}^*$ are M-V systems implies that the other is  also a M-V system.

$(4')$ If $F_\alpha:\mathcal{N}^*\rightarrow \mathcal{M}_\alpha^*$ (resp. $G_\alpha:\mathcal{M}_\alpha^*\rightarrow \mathcal{N}^*$) are morphisms of M-V systems of  presheaves (resp. precosheaves) for all $\alpha\in \Lambda$, then
\begin{displaymath}
(F^*_\alpha)_{\alpha\in\Lambda}:\mathcal{N}^*\rightarrow\prod_{\alpha\in\Lambda}\mathcal{M}^*_\alpha
\end{displaymath}
(resp.
\begin{displaymath}
\sum_{\alpha\in\Lambda}G^*_\alpha:\bigoplus_{\alpha\in\Lambda}\mathcal{M}^*_\alpha\rightarrow\mathcal{N}^*)
\end{displaymath}
is a  morphism of M-V systems of   presheaves (resp. precosheaves).

Assume that $\Lambda$ is a finite set. Then $(F^*_\alpha)_{\alpha\in\Lambda}$ and $\sum_{\alpha\in\Lambda}G^*_\alpha$ are both  morphisms of M-V systems of  presheaves (resp.  precosheaves).

$(5')$ If $F^*:\mathcal{M}^*\rightarrow \mathcal{N}^*$ and $G^*:\mathcal{L}^*\rightarrow \mathcal{M}^*$ are both  morphisms of M-V morphisms of presheaves (resp. precosheaves), then the composition $F^*\circ G^*$ is a M-V morphism.

$(6')$ If $F^*_i:\mathcal{M}^*\rightarrow \mathcal{N}^*$ are morphisms of  M-V systems of presheaves (resp.  precosheaves) for all  $1\leq i\leq n$, then $\sum_{i=1}^nF^*_i:\mathcal{M}^*\rightarrow \mathcal{N}^*$ is a M-V morphism, where $(\sum_{i=1}^nF^p_i)(\alpha)=\sum_{i=1}^nF^p_i(\alpha)$ for any $p\in\mathbb{Z}$, open set $U$ and $\alpha\in \mathcal{M}^p(U)$.

$(7')$ If $F_\alpha^*:\mathcal{M}_\alpha^*\rightarrow\mathcal{N}_\alpha^*$ are  morphisms of M-V systems of presheaves  (resp. precosheaves) for all $\alpha\in \Lambda$, then
\begin{displaymath}
\bigoplus_{\alpha\in\Lambda}F_\alpha^*:\bigoplus_{\alpha\in\Lambda}\mathcal{M}_\alpha^*\rightarrow\bigoplus_{\alpha\in\Lambda}\mathcal{N}_\alpha^*
\end{displaymath}
and
\begin{displaymath}
\prod_{\alpha\in\Lambda}F_\alpha^*:\prod_{\alpha\in\Lambda}\mathcal{M}_\alpha^*\rightarrow\prod_{\alpha\in\Lambda}\mathcal{N}_\alpha^*
\end{displaymath}
are both M-V morphisms.

\subsection{A proof of Theorem \ref{1.1}}
First, recall the \emph{glued principle}, which was proved in our previous articles.  For readers' convenience, we give a complete proof here.
\begin{lem}[\cite{M2}]\label{glued}
Denote by $\mathcal{P}(X)$ a statement on a smooth manifold $X$. Assume that $\mathcal{P}$ satisfies  conditions:

$(i)$ \emph{(\emph{local condition})} There exists a basis $\mathfrak{U}$ of topology of $X$, such that, $\mathcal{P}(U_1\cap...\cap U_l)$ holds for any finite $U_1,..., U_l\in \mathfrak{U}$.

$(ii)$ \emph{(disjoint condition)} Let $\{U_n|n\in\mathbb{N}^+\}$ be any collection of disjoint open subsets of $X$. If $\mathcal{P}(U_n)$ hold for all $n\in\mathbb{N}^+$, $\mathcal{P}(\bigcup_{n=1}^\infty U_n)$ holds.

$(iii)$ \emph{(Mayer-Vietoris condition)} For open subsets $U$, $V$ of $X$, if $\mathcal{P}(U)$, $\mathcal{P}(V)$ and $\mathcal{P}(U\cap V)$ hold, then $\mathcal{P}(U\cup V)$ holds.\\
Then $\mathcal{P}(X)$ holds.
\end{lem}
\begin{proof}
We prove the follows:

$(c.1)$ For open subsets  $U_1$, $\ldots$, $U_r$ of $X$, if $\mathcal{P}(U_{i_1}\cap\ldots \cap U_{i_k})$ holds for any $1\leq i_1<\ldots<i_k\leq r$, then $\mathcal{P}(\bigcup_{i=1}^r U_i)$ holds.

Obviously, it holds for $r=1$. Suppose $(c.1)$ holds for $r$. For $r+1$, set
$U'_1=U_1$, $\ldots$, $U'_{r-1}=U_{r-1}$, $U'_r=U_r\cup U_{r+1}$. Then $\mathcal{P}(U'_{i_1}\cap\ldots \cap U'_{i_k})$ holds for any $1\leq i_1<\ldots<i_k\leq r-1$. By the Mayer-Vietoris condition, $\mathcal{P}(U'_{i_1}\cap\ldots \cap U'_{i_{k-1}}\cap U'_r)$ also holds for any $1\leq i_1<\ldots<i_{k-1}\leq r-1$, since $\mathcal{P}(U_{i_1}\cap\ldots \cap U_{i_{k-1}}\cap U_r)$, $\mathcal{P}(U_{i_1}\cap\ldots \cap U_{i_{k-1}}\cap U_{r+1})$ and  $\mathcal{P}(U_{i_1}\cap\ldots \cap U_{i_{k-1}}\cap U_r\cap U_{r+1})$  hold. By the inductive hypothesis, $\mathcal{P}(\bigcup_{i=1}^{r+1} U_i)=\mathcal{P}(\bigcup_{i=1}^{r} U'_i)$ holds. We proved $(c.1)$.

Let $\mathfrak{U}_\mathfrak{f}$ be the collection of open sets  which are finite unions of open sets in  $\mathfrak{U}$. We claim that

$(c.2)$ $\mathcal{P}(V)$ holds for any finite intersection $V$ of open sets in $\mathfrak{U}_\mathfrak{f}$.

Set $V=\bigcap_{i=1}^s U_i$, where $U_i=\bigcup_{j=1}^{r_i}U_{ij}$ and $U_{ij}\in\mathfrak{U}$. Then $V=\bigcup_{J\in\Lambda}U_J$, where $\Lambda=\{J=(j_1,...,j_s)|1\leq j_1\leq r_1,\ldots,1\leq j_s\leq r_s\}$ and $U_J=U_{1j_1}\cap...\cap U_{sj_s}$. For any $J_1,\ldots,J_t\in\Lambda$, $\mathcal{P}(U_{J_1}\cap \ldots\cap U_{J_t})$ holds by the local condition. Hence $\mathcal{P}(V)=\mathcal{P}(\bigcup_{J\in\Lambda}U_J)$ holds by $(c.1)$.

By \cite{GHV}, p. 16, Prop. II, $X=V_1\cup...\cup V_l$, where $V_i$ is a countable disjoint union of open sets in  $\mathfrak{U}_\mathfrak{f}$. Obviously, for any $1\leq i_1<\ldots<i_k\leq l$, $V_{i_1}\cap \ldots \cap V_{i_k}$ is a countable  disjoint union of the finite intersection of open sets in $\mathfrak{U}_\mathfrak{f}$. By the disjoint condition and $(c.2)$,  $\mathcal{P}(V_{i_1}\cap \ldots \cap V_{i_k})$ holds. Then $\mathcal{P}(X)$ holds by $(c.1)$.
\end{proof}

Now, we give a proof of Theorem \ref{1.1}.

\begin{proof}
We only prove the case of presheaves. The other case can be proved similarly.

Denote by $\mathcal{P}(U)$ the statement that $F^*(U)$ is an isomorphism.
By the hypothesis (*), $\mathcal{P}(\bigcap_{i=1}^l U_i)$ hold for any finite open sets $U_1$, $\ldots$, $U_l\in \mathfrak{U}$, i.e., $\mathcal{P}$ satisfies the local condition in Lemma \ref{glued}. For any $p\in\mathbb{Z}$, $F^p$ is a morphism of cdp presheaves, so $\mathcal{P}$  satisfies the disjoint condition in Lemma \ref{glued}. Since $F^*$ is a M-V morphism,  we have a commutative diagram of long exact sequences
\begin{displaymath}
\tiny{
\xymatrix{
  \cdots\mathcal{M}^{p-1}(U\cap V)\ar[d]^{F^{p-1}(U\cap V)}\ar[r]^{\quad\delta_{U,V}^{p-1}}& \mathcal{M}^p(U\cup V)\ar[d]^{F^{p}(U\cup V)}\ar[r]^{P^p\quad}&\mathcal{M}^p(U)\oplus \mathcal{M}^p(V)\ar[d]^{(F^p(U),F^p(V))}\ar[r]^{\quad Q^p}&\mathcal{M}^{p}(U\cap V)\ar[d]^{F^{p}(U\cap V)}\ar[r]^{\delta_{U,V}^{p}\quad}&\mathcal{M}^{p+1}(U\cup V)\ar[d]^{F^{p+1}(U\cup V)}\cdots\\
 \cdots\mathcal{N}^{p-1}(U\cap V)\ar[r]^{\quad\delta_{U,V}^{p-1}}& \mathcal{N}^p(U\cup V)\ar[r]^{P^p\quad}&\mathcal{N}^p(U)\oplus \mathcal{N}^p(V)\ar[r]^{\quad Q^p}&\mathcal{N}^{p}(U\cap V)\ar[r]^{\delta_{U,V}^{p}\quad}&\mathcal{N}^{p+1}(U\cup V)\cdots, }}
\end{displaymath}
for open sets $V\subseteq U$. If $F^*(U)$, $F^*(V)$ and $F^*(U\cap V)$ are isomorphisms, then $F^*(U\cup V)$ is an isomorphism by the five-lemma.  Therefore, $\mathcal{P}$  satisfies the Mayer-Vietoris condition. By Lemma \ref{glued}, $F^*(X)$ is an isomorphism.

For any open set $V$ in $X$, set $\mathfrak{U}_V=\{U\in\mathfrak{U}|U\subseteq V\}$, which is a basis of topology of $V$. By the hypothesis (*),  $F^*(U_1\cap...\cap U_l)$ is isomorphic  for any finite $U_1$, $\ldots$, $U_l\in \mathfrak{U}_V\subseteq \mathfrak{U}$. So $F^*(V)$ is an isomorphism as above. We complete the proof.
\end{proof}

\section{Several Operators}
Before giving examples of Mayer-Vietoris systems, we define several operators on forms and currents with values in local systems and locally free sheaves, which may be well known for experts. We do not find appropriate references on them, and hence, give all details here by the sheaf-theoretic approach completely.  Part of them was defined in a different viewpoint in \cite{Se,W}, where they used the language of bundle-valued forms and their topological duals (called bundle-valued currents).
\subsection{Complex manifolds and locally free sheaves}
Let $X$ be a connected complex manifold and $\mathcal{E}$ a locally free sheaf of $\mathcal{O}_X$-modules of rank $m$ on $X$. An open subset $U$ of $X$ is said to be \emph{$\mathcal{E}$-free}, if the restriction $\mathcal{E}|_U$ is a free sheaf of $\mathcal{O}_U$-modules. An open covering $\mathfrak{U}$ of $X$ is said to be \emph{$\mathcal{E}$-free}, if all $U\in\mathfrak{U}$ are $\mathcal{E}$-free. For an open set $U\subseteq X$, the elements of $\Gamma(U, \mathcal{E}\otimes_{\mathcal{O}_X}\mathcal{A}_X^{p,q})$ and $\Gamma(U, \mathcal{E}\otimes_{\mathcal{O}_X}\mathcal{D}_X^{\prime p,q})$ are called \emph{$\mathcal{E}$-vlaued $(p,q)$-forms and currents on $U$}, respectively. From now on, all tensors $\otimes_{\mathcal{O}_X}$ of sheaves of $\mathcal{O}_X$-modules will be simply denoted by $\otimes$.
\subsubsection{\emph{\textbf{Local representations}}}
Let $U$ be a $\mathcal{E}$-free open subset of $X$ and  $e_1$, $\ldots$, $e_m$  a basis of $\Gamma(U,\mathcal{E})$ as an $\mathcal{O}_X(U)$-module.  For $\omega\in\Gamma(X, \mathcal{E}\otimes\mathcal{A}_X^{p,q})$, the restriction $\omega|_U$ to $U$ can be written as $\sum_{i=1}^{m}e_i\otimes \alpha_i$, where $\alpha_1$, $\ldots$, $\alpha_m\in \mathcal{A}^{p,q}(U)$. Similarly, for $S\in\Gamma(X, \mathcal{E}\otimes\mathcal{D}_X^{\prime p,q})$, $S|_U=\sum_{i=1}^{m}e_i\otimes T_i$, where $T_1$, $\ldots$, $T_m\in \mathcal{D}^{\prime p,q}(U)$. We easily get
\begin{lem}\label{support}
For any $1\leq i\leq m$, $\emph{supp}\alpha_i\subseteq\emph{supp}\omega\cap U$ and $\emph{supp}T_i\subseteq\emph{supp}S\cap U$.
\end{lem}
\subsubsection{\emph{\textbf{Extensions by zero and restrictions}}}
Assume that $j:V\rightarrow X$ is the inclusion of an open subset $V$ into $X$. Let $\mathfrak{U}$ be a $\mathcal{E}$-free covering of $X$ and $e^U_1$, $\ldots$, $e^U_m$  a basis of $\Gamma(U,\mathcal{E})$ as an $\mathcal{O}_X(U)$-module for any $U\in\mathfrak{U}$.

For any $\omega\in \Gamma_c(V,\mathcal{E}\otimes\mathcal{A}_X^{p,q})$ and $U\in\mathfrak{U}$, the restriction $\omega|_{V\cap U}$ to $U\cap V$ is $\sum_{i=1}^me_i^U|_{V\cap U}\otimes \alpha_i$, where $\alpha_1$, $\ldots$, $\alpha_m\in \mathcal{A}^{p,q}(V\cap U)$. Clearly, $\alpha_i=0$ on $(V\cap U)\cap (U-\textrm{supp}\alpha_i)=V\cap U-\textrm{supp}\alpha_i$. So $\alpha_i$ can be extended on $(V\cap U)\cup (U-\textrm{supp}\alpha_i)=U$ by zero, denoted by $\tilde{\alpha}_i$.  Set
\begin{displaymath}
\widetilde{\omega}_U=\sum_{i=1}^me^U_i\otimes \tilde{\alpha}_i
\end{displaymath}
in $\Gamma(U,\mathcal{E}\otimes\mathcal{A}_X^{p,q})$. Then $\{\widetilde{\omega}_U|U\in\mathfrak{U}\}$ can be glued as a global section of $\mathcal{E}\otimes\mathcal{A}_X^{p,q}$ on $X$, denoted by $j_*\omega$. It is noteworthy that $j_*\omega$ does not depend on the choice of the $\mathcal{E}$-free open covering $\mathfrak{U}$. Actually, let $\mathfrak{U}_0$ be the collection of all $\mathcal{E}$-free open sets in $X$. Any  $\mathcal{E}$-free covering $\mathfrak{U}$ is a subcovering of $\mathfrak{U}_0$, then $\{\tilde{\omega}_U|U\in\mathfrak{U}\}\subseteq \{\tilde{\omega}_U|U\in\mathfrak{U}_0\}$. So $j_*\omega$ defined by $\mathfrak{U}$ and $\mathfrak{U}_0$ coincides. Since $\textrm{supp}j_*\omega=\textrm{supp}\omega$ is compact, we get a map
\begin{equation}\label{extension}
j_*:\Gamma_c(V,\mathcal{E}\otimes\mathcal{A}_X^{p,q})\rightarrow\Gamma_c(X,\mathcal{E}\otimes\mathcal{A}_X^{p,q}),
\end{equation}
which is exactly \emph{the extension  by zero} of sections of the sheaf $\mathcal{E}\otimes\mathcal{A}_X^{p,q}$ with compact supports.

Let $j^*:\Gamma(X,\mathcal{E}\otimes\mathcal{D}_X^{\prime p,q})\rightarrow\Gamma(V,\mathcal{E}\otimes\mathcal{D}_X^{\prime p,q})$ be the \emph{restriction} of the sheaf $\mathcal{E}\otimes\mathcal{D}_X^{\prime p,q}$. For any $S\in \Gamma(X,\mathcal{E}\otimes\mathcal{D}_X^{\prime p,q})$ and $U\in\mathfrak{U}$, if the restriction $S|_U=\sum_{i=1}^{m}e_i^U\otimes T_i$, where $T_1$, $\ldots$, $T_m\in \mathcal{D}^{\prime p,q}(U)$, then $(j^*S)|_{V\cap U}=\sum_{i=1}^{m}e^U_i|_{V\cap U}\otimes T_i|_{V\cap U}$.

\subsubsection{\emph{\textbf{Pullbacks and pushouts}}}
Let $f:Y\rightarrow X$ be a holomorphic map of connected complex manifolds and $r=\textrm{dim}_{\mathbb{C}}Y-\textrm{dim}_{\mathbb{C}}X$. Put $f^*\mathcal{E}=f^{-1}\mathcal{E}\otimes_{f^{-1}\mathcal{O}_X}\mathcal{O}_Y$ the inverse image of $\mathcal{E}$ by $f$. The adjunction morphism $\mathcal{E}\rightarrow f_*f^*\mathcal{E}$ induces $f_U^*:\Gamma(U,\mathcal{E})\rightarrow \Gamma(f^{-1}(U),f^*\mathcal{E})$ for any open set $U\subseteq X$, where $f_U:f^{-1}(U)\rightarrow U$ is the restriction of $f$ to $f^{-1}(U)$.

\textbf{Pullbacks.}
The pullback $\mathcal{A}_X^{p,q}\rightarrow f_*\mathcal{A}_Y^{p,q}$ induces a morphism of sheaves
\begin{displaymath}
\mathcal{E}\otimes\mathcal{A}_X^{p,q}\rightarrow \mathcal{E}\otimes f_*\mathcal{A}_Y^{p,q}=f_*(f^*\mathcal{E}\otimes\mathcal{A}_Y^{p,q}),
\end{displaymath}
hence induces a \emph{pullback} of $\mathcal{E}$-valued $(p,q)$-forms
\begin{equation}\label{pullback}
f^*:\Gamma(X,\mathcal{E}\otimes\mathcal{A}_X^{p,q})\rightarrow\Gamma(Y,f^*\mathcal{E}\otimes\mathcal{A}_Y^{p,q}).
\end{equation}
Suppose that  $U$  is a $\mathcal{E}$-free  open set in $X$ and $e_1$, $\ldots$, $e_m$ is  a basis of $\Gamma(U,\mathcal{E})$ as an $\mathcal{O}_X(U)$-module.
Obviously, $f_U^*e_1$, $\ldots$, $f_U^*e_m$ is  a basis of $\Gamma(f^{-1}(U),f^*\mathcal{E})$ as an $\mathcal{O}_Y(f^{-1}(U))$-module. For a $\mathcal{E}$-valued $(p,q)$-form $\omega$, set $\omega|_U=\sum_{i=1}^{m}e_i\otimes \alpha_i$, where $\alpha_1$, $\ldots$, $\alpha_m\in \mathcal{A}^{p,q}(U)$.  Then
\begin{equation}\label{pullback-rep}
(f^*\omega)|_{f^{-1}(U)}=\sum_{i=1}^{m}f_U^*e_i\otimes f_U^*\alpha_i.
\end{equation}
We have
\begin{lem}\label{pullback-support}
$\emph{supp}f^*\omega\subseteq f^{-1}(\emph{supp}\omega)$.
\end{lem}
\begin{proof}
For any $y\in Y- f^{-1}(\textrm{supp}\omega)$, $f(y)\in X-\textrm{supp}\omega$. There exists a $\mathcal{E}$-free open neighborhood $V$ of $f(y)$ such that $V\subseteq X- \textrm{supp}\omega$, and then $\omega|_V=0$. By the local representation (\ref{pullback-rep}) of $f^*\omega$, $(f^*\omega)|_{f^{-1}(V)}=0$, i.e., $\textrm{supp}f^*\omega\cap f^{-1}(V)=\emptyset$. So $y$ is not in $\textrm{supp}f^*\omega$. We proved the first part.
\end{proof}
If $f$ is \emph{proper}, the pullback (\ref{pullback})  gives
\begin{displaymath}
f^*:\Gamma_c(X,\mathcal{E}\otimes\mathcal{A}_X^{p,q})\rightarrow\Gamma_c(Y,f^*\mathcal{E}\otimes\mathcal{A}_Y^{p,q}),
\end{displaymath}
by Lemma \ref{pullback-support}.

\textbf{Pushouts.}
Assume that $S$ is a $f^*\mathcal{E}$-valued $(p,q)$-current on $Y$  satisfying that $f|_{\textrm{supp}S}:\textrm{supp}S\rightarrow X$ is \emph{proper}. Let $\mathfrak{U}$ be a  $\mathcal{E}$-free  covering of $X$ and  $e^U_1$, $\ldots$, $e^U_m$  a basis of $\Gamma(U,\mathcal{E})$ as an $\mathcal{O}_X(U)$-module for any $U\in\mathfrak{U}$.   For $U\in\mathfrak{U}$, $S$ can be written as $\sum_{i=1}^{m}f_U^*e^U_i\otimes T_i$ on $f^{-1}(U)$, where $T_1$, $\ldots$, $T_m\in \mathcal{D}^{\prime p,q}(f^{-1}(U))$.  By Lemma \ref{support}, $\textrm{supp}T_i\subseteq\textrm{supp}S\cap f^{-1}(U)$, and then $f_U|_{\textrm{supp}T_i}:\textrm{supp}T_i\rightarrow U$ is proper. So $f_{U*}T_i$ is well defined. We get  a $\mathcal{E}$-valued $(p-r,q-r)$-current
\begin{equation}\label{pushout}
\widetilde{S}_U=\sum_{i=1}^{m}e^U_i\otimes f_{U*}T_i
\end{equation}
on $U$. For $U$, $U^\prime\in\mathfrak{U}$ satisfying $U\cap U^\prime\neq\emptyset$, $\widetilde{S}_U=\widetilde{S}_{U^\prime}$ on $U\cap U^\prime$. Hence we get a $\mathcal{E}$-valued $(p-r,q-r)$-current on $X$ , denoted by $f_*S$, such that $(f_*S)|_U=\widetilde{S}_U$.  Similarly with $j_*$, the definition of  $f_*S$ is independent of the choice of $\mathcal{E}$-free coverings.

\begin{lem}\label{pushout-support}
If $f|_{\emph{supp}S}:\emph{supp}S\rightarrow X$ is proper, then $\emph{supp}f_*S\subseteq f(\emph{supp}S)$.
\end{lem}
\begin{proof}
For any $x\in X- f(\textrm{supp}S)$, there exists a $\mathcal{E}$-free open neighborhood $U$ of $x$ such that $U\subseteq X- f(\textrm{supp}S)$. Clearly, $f^{-1}(U)\cap\textrm{supp}S=\emptyset$, and then $S|_{f^{-1}(U)}=0$. By the definition (\ref{pushout}) of $f_*S$, $(f_*S)|_U=0$, i.e., $\textrm{supp}f_*S\cap U=\emptyset$. Hence, $x$ is not in $\textrm{supp}f_*S$.  The second part was proved.
\end{proof}
By Lemma \ref{pushout-support}, we get a \emph{pushout}
\begin{displaymath}
f_*:\Gamma_c(Y,f^*\mathcal{E}\otimes\mathcal{D}_Y^{\prime p,q})\rightarrow\Gamma_c(X,\mathcal{E}\otimes\mathcal{D}_X^{\prime p-r,q-r}).
\end{displaymath}
In particular, for  the inclusion $j:V\rightarrow X$ of an open subset $V$ of $X$, $j_*$ is the extension by zero, whose restriction to $\Gamma_c(V,\mathcal{E}\otimes\mathcal{A}_X^{p,q})$  concides with (\ref{extension}). If $f$ is \emph{proper},  by Lemma \ref{pushout-support}, (\ref{pushout}) defines a \emph{pushout}
\begin{displaymath}
f_*:\Gamma(Y,f^*\mathcal{E}\otimes\mathcal{D}_Y^{\prime p,q})\rightarrow\Gamma(X,\mathcal{E}\otimes\mathcal{D}_X^{\prime p-r,q-r}),
\end{displaymath}
which is actually induced by the pushout $f_*\mathcal{D}_Y^{\prime p,q}\rightarrow \mathcal{D}_X^{\prime p-r,q-r}$.

\begin{prop}\label{com1}
Suppose that $f:Y\rightarrow X$ is a proper holomorphic map of connected complex manifolds and $\mathcal{E}$ is a locally free sheaf of $\mathcal{O}_X$-modules of finite rank on $X$. For an open set $V$ of $X$, let $f_V:f^{-1}(V)\rightarrow V$ be the restriction of $f$ to $f^{-1}(V)$ and let  $j:V\rightarrow X$ and $j':f^{-1}(V)\rightarrow Y$  be inclusions. Then, $j'_*f_V^*=f^*j_*$ on $\Gamma_c(V,\mathcal{E}\otimes\mathcal{A}_X^{*,*})$ and $f_{V*}j'^*=j^*f_*$ on $\Gamma(Y,f^*\mathcal{E}\otimes\mathcal{D}_Y^{\prime*,*})$.
\end{prop}
\begin{proof}
Let $U$ be any $\mathcal{E}$-free open set of $X$ and  $e_1$, $\ldots$, $e_m$  a basis of $\Gamma(U,\mathcal{E})$ as an $\mathcal{O}_X(U)$-module.

For $\omega\in\Gamma_c(V,\mathcal{E}\otimes\mathcal{A}_X^{*,*})$, $\omega|_{V\cap U}=\sum_{i=1}^me_i|_{V\cap U}\otimes\alpha_i$,
where $\alpha_1$, $\ldots$, $\alpha_m\in\mathcal{A}^{*,*}(V\cap U)$. By defintions,
\begin{displaymath}
(j_*\omega)|_{U}=\sum_{i=1}^me_i\otimes \tilde{\alpha}_i,
\end{displaymath}
\begin{displaymath}
(f^*j_*\omega)|_{f^{-1}(U)}=\sum_{i=1}^mf_U^*e_i\otimes f_U^{*}\tilde{\alpha}_i,
\end{displaymath}
\begin{displaymath}
\begin{aligned}
(f_V^{*}\omega)|_{f^{-1}(V\cap U)}=&\sum_{i=1}^mf_{V\cap U}^*(e_i|_{V\cap U})\otimes f_{V\cap U}^{*}\alpha_i\\
=&\sum_{i=1}^m(f_U^*e_i)|_{f^{-1}(V\cap U)}\otimes f_{V\cap U}^{*}\alpha_i,
\end{aligned}
\end{displaymath}
\begin{displaymath}
(j^{\prime}_*f_V^{*}\omega)|_{f^{-1}(U)}=\sum_{i=1}^mf_U^*e_i\otimes \widetilde{f_{V\cap U}^{*}\alpha_i},
\end{displaymath}
where $\tilde{\alpha}_i\in\mathcal{A}^{*,*}(U)$ and $\widetilde{f_{V\cap U}^{*}\alpha_i}\in\mathcal{A}^{*,*}(f^{-1}(U))$ denote the extensions of $\alpha_i$ and $f_{V\cap U}^{*}\alpha_i$ by zero, respectively. Since $f_U^{*}\tilde{\alpha}_i=\widetilde{f_{V\cap U}^{*}\alpha_i}$,  $(f^*j_*\omega)|_{f^{-1}(U)}=(j^{\prime}_*f_V^*\omega)|_{f^{-1}(U)}$. We proved the first part.

For $S\in\Gamma(Y,f^*\mathcal{E}\otimes\mathcal{D}_Y^{\prime*,*})$, $S|_{f^{-1}(U)}=\sum_{i=1}^{m}f_U^*e_i\otimes T_i$,
where $T_1$, $...$, $T_m\in \mathcal{D}^{\prime p,q}(f^{-1}(U))$. Then
\begin{displaymath}
\begin{aligned}
(j^\prime_*S)|_{f^{-1}(V\cap U)}=&\sum_{i=1}^m(f_U^*e_i)|_{f^{-1}(V\cap U)}\otimes T_i|_{f^{-1}(V\cap U)}\\
=&\sum_{i=1}^mf_{V\cap U}^*(e_i|_{V\cap U})\otimes T_i|_{f^{-1}(V\cap U)},
\end{aligned}
\end{displaymath}
\begin{displaymath}
(f_{V*}j^\prime_*S)|_{V\cap U}=\sum_{i=1}^me_i|_{V\cap U}\otimes f_{V\cap U*}(T_i|_{f^{-1}(V\cap U)}),
\end{displaymath}
\begin{displaymath}
(f_*S)|_{U}=\sum_{i=1}^{m}e_i\otimes f_{U*}T_i,
\end{displaymath}
\begin{displaymath}
(j^*f_*S)|_{V\cap U}=\sum_{i=1}^{m}e_i|_{V\cap U}\otimes (f_{U*}T_i)|_{V\cap U}.
\end{displaymath}
It is easily to check that $f_{V\cap U*}(T_i|_{f^{-1}(V\cap U)})=(f_{U*}T_i)|_{V\cap U}$, hence $(f_{V*}j^\prime_*S)|_{V\cap U}=(j^*f_*S)|_{V\cap U}$. We complete the proof.
\end{proof}

\subsubsection{\emph{\textbf{Operators on cohomology}}}
We still denote by $\overline{\partial}$ the differentials $1\otimes \overline{\partial}:\mathcal{E}\otimes\mathcal{A}_X^{p,q}\rightarrow \mathcal{E}\otimes\mathcal{A}_X^{p,q+1}$ and $\mathcal{E}\otimes\mathcal{D}_X^{\prime p,q}\rightarrow \mathcal{E}\otimes\mathcal{D}_X^{\prime p,q+1}$.
Then $\mathcal{E}\otimes\Omega_X^p$ has two soft resolutions
\begin{displaymath}
\xymatrix{
0\ar[r] &\mathcal{E}\otimes\Omega_X^p\ar[r]^{i} &\mathcal{E}\otimes\mathcal{A}_{X}^{p,0}\ar[r]^{\quad\quad\bar{\partial}}&\cdots\ar[r]^{\bar{\partial}\quad\quad}&\mathcal{E}\otimes\mathcal{A}_{X}^{p,n}\ar[r]&0
}
\end{displaymath}
and
\begin{displaymath}
\xymatrix{
0\ar[r] &\mathcal{E}\otimes\Omega_X^p\ar[r]^{i} &\mathcal{E}\otimes\mathcal{D}_{X}^{\prime p,0}\ar[r]^{\quad\quad\bar{\partial}}&\cdots\ar[r]^{\bar{\partial}\quad\quad}&\mathcal{E}\otimes\mathcal{D}_{X}^{\prime p,n}\ar[r]&0
},
\end{displaymath}
where $\textrm{dim}_{\mathbb{C}}X=n$. So
\begin{displaymath}
H^q(X,\mathcal{E}\otimes\Omega_X^p)=H^q(\Gamma(X,\mathcal{E}\otimes\mathcal{A}_{X}^{p,\bullet}))=H^q(\Gamma(X,\mathcal{E}\otimes\mathcal{D}_{X}^{\prime p,\bullet})),
\end{displaymath}
which are called \emph{$\mathcal{E}$-valued Dolbeault cohomology} and denoted by  $H^{p,q}(X,\mathcal{E})$. They coincides with the bundle-valued Dolbeault cohomology $H^{p,q}(X,E)$ (seeing \cite{Dem}, p. 268, Prop. 11.5), where $E$ is the holomorphic vector bundle associated to $\mathcal{E}$.  Similarly, denote by $H_c^{p,q}(X,\mathcal{E})$ the $\mathcal{E}$-valued Dolbeault cohomology with compact support. Clearly, all operators defined in Section 3.1.2 and 3.1.3 commutate with $\bar{\partial}$, hence induce the corresponding morphisms at the level of cohomology.
\subsubsection{\emph{\textbf{Wedge and cup products}}}
Assume that $\mathcal{E}$ and $\mathcal{F}$ are locally free sheaves of $\mathcal{O}_X$-modules of rank $m$ and $n$ on $X$ respectively and $\mathfrak{U}$ is an open covering of $X$ satisfying that any $U\in\mathfrak{U}$ is $\mathcal{E}$-free and $\mathcal{F}$-free both. Let $e^U_1$, $\ldots$, $e^U_m$ and $f^U_1$, $\ldots$, $f^U_n$ be bases of $\Gamma(U,\mathcal{E})$ and $\Gamma(U, \mathcal{F})$ as $\mathcal{O}_X(U)$-modules, respectively.

For $S\in\Gamma(X, \mathcal{E}\otimes\mathcal{D}_X^{\prime p,q})$, $\omega\in\Gamma(X, \mathcal{F}\otimes\mathcal{A}_X^{r,s})$ and $U\in\mathfrak{U}$, $S$ and $\omega$  are represented by $\sum_{i=0}^{m}e^U_i\otimes T_i$ and $\sum_{i=0}^{n}f^U_i\otimes \alpha_i$ on $U$ respectively, where $T_i\in\mathcal{D}^{\prime p,q}(U)$ and $\alpha_i\in\mathcal{A}^{r,s}(U)$ for any $i$. Then
\begin{displaymath}
\sum_{1\leq i\leq m}\sum_{1\leq j\leq n} e^U_i\otimes f^U_j\otimes (T_i\wedge\alpha_j)
\end{displaymath}
gives a $\mathcal{E}\otimes\mathcal{F}$-valued $(p+r,q+s)$-current on $U$. They are glued as a global section of  $ \mathcal{E}\otimes\mathcal{F}\otimes\mathcal{D}_X^{\prime p+r,q+s}$ on $X$, which is independent of the choice of open coverings. It is called the \emph{wedge product} of $S$ and $\omega$, denoted by $S\wedge\omega$. Similarly, the $\mathcal{F}\otimes\mathcal{E}$-valued current $\omega\wedge S$ and the $\mathcal{E}\otimes\mathcal{F}$-valued form $\psi\wedge \omega$  on $X$ can be defined well, where $\psi\in\Gamma(X, \mathcal{E}\otimes\mathcal{A}_X^{p,q})$. Clearly,
\begin{displaymath}
\overline{\partial}(S\wedge\omega)=\overline{\partial}S\wedge\omega+(-1)^{p+q}S\wedge \overline{\partial}\omega,
\end{displaymath}
\begin{displaymath}
f^*(\psi\wedge\omega)=f^*\psi\wedge f^*\omega,
\end{displaymath}
where $f:Y\rightarrow X$ is a holomorphic map of complex manifolds. In addition, if $T$ is a $f^*\mathcal{E}$-valued current on $Y$ satisfying that $f|_{\textrm{supp}T}:\textrm{supp}T\rightarrow X$ is proper, then
\begin{equation}\label{pro-formula1}
f_*(T\wedge f^*\omega)=f_*T\wedge\omega,
\end{equation}
which is called the \emph{projection formula}. Actually, it is the classical projection formula of forms and currents locally.

Define the \emph{cup product} on cohomology groups
\begin{displaymath}
\cup:H^{p,q}(X,\mathcal{E})\times H^{r,s}(X,\mathcal{F})\rightarrow H^{p+r,q+s}(X,\mathcal{E}\otimes\mathcal{F})
\end{displaymath}
as $[\psi]\cup[\omega]=[\psi\wedge\omega]$ or $[S]\cup[\omega]=[S\wedge\omega]$. Similarly, we can define the cup products between $H^{p,q}(X,\mathcal{E})$ or $H_c^{p,q}(X,\mathcal{E})$ and $H^{r,s}(X,\mathcal{F})$ or $H_c^{r,s}(X,\mathcal{F})$.
By the projection formula (\ref{pro-formula1}),
\begin{equation}\label{pro-formula2}
f_*(\varphi\cup f^*\eta)=f_*\varphi\cup\eta,
\end{equation}
for $\varphi\in H_c^{*,*}(Y,f^*\mathcal{E})$, $\eta\in H^{*,*}(X,\mathcal{F})$. Moreover, if $f$ is proper, they also hold for  $\varphi\in H^{*,*}(Y,f^*\mathcal{E})$, $\eta\in H^{*,*}(X,\mathcal{F})$, or $\varphi\in H_c^{*,*}(Y,f^*\mathcal{E})$, $\eta\in H_c^{*,*}(X,\mathcal{F})$, or $\varphi\in H^{*,*}(Y,f^*\mathcal{E})$, $\eta\in H_c^{*,*}(X,\mathcal{F})$.

\subsection{Smooth manifolds and local systems}
Recall that, for a topology space $X$, \emph{a local system of $\mathbb{R}$-modules} on $X$ is a locally constant sheaf of $\mathbb{R}$-modules on $X$, or equivalently, a locally free sheaf of $\underline{\mathbb{R}}_X$-modules on $X$, where $\underline{\mathbb{R}}_X$ is the constant sheaf with stalks $\mathbb{R}$ on $X$. Assume that $\mathcal{V}$ is a local system of $\mathbb{R}$-modules on $X$. An open subset $U$ of $X$ is said to be \emph{$\mathcal{V}$-constant}, if the restriction $\mathcal{V}|_U$ is a contant sheaf. An open covering $\mathfrak{U}$ of $X$ is said to be \emph{$\mathcal{V}$-constant}, if all $U\in\mathfrak{U}$ are $\mathcal{V}$-constant.

For a smooth map $f:Y\rightarrow X$ and  local systems $\mathcal{V}$, $\mathcal{W}$ of $\mathbb{R}$-modules of finite ranks on $X$, all notions in Section 3.1 can be similarly defined and the corresponding results there also hold, where we only need to  replace  $\mathcal{O}_X$,  $\mathcal{E}$, $\mathcal{E}$-free, $\mathcal{A}_X^{*,*}$, $\mathcal{D}_X^{\prime *,*}$, $f^*\mathcal{E}$,  $H^{*,*}(X,\mathcal{E})$, $H_c^{*,*}(X,\mathcal{E})$ and $\mathcal{E}\otimes_{\mathcal{O}_X}\mathcal{F}$ with $\underline{\mathbb{R}}_X$, $\mathcal{V}$, $\mathcal{V}$-constant, $\mathcal{A}_X^{*}$, $\mathcal{D}_X^{\prime *}$, $f^{-1}\mathcal{V}$,  $H^*(X,\mathcal{V})$, $H_c^*(X,\mathcal{V})$ and $\mathcal{V}\otimes_{\underline{\mathbb{R}}_X}\mathcal{W}$, respectively. It is noteworthy that, if the definitions of notions involve the currents, then the related manifolds are necessarily oriented. Moreover, all notions and results in this subsection are still valid, if we use $\mathbb{C}$ instead of $\mathbb{R}$.

We list out some results as follows.
\begin{lem}\label{support2}
Suppose that $X$ is a connected smooth manifold and $\mathcal{V}$ is a local system of $\mathbb{R}$-modules of rank $m$ on $X$. Let $U$ be a $\mathcal{V}$-constant open subset of $X$ and  $v_1$, $\ldots$, $v_m$  a basis of $\Gamma(U,\mathcal{V})$.

$(1)$ For $\omega\in\Gamma(X, \mathcal{V}\otimes\mathcal{A}_X^{p})$, if the restriction $\omega|_U$ to $U$ is written as $\sum_{i=1}^{m}v_i\otimes \alpha_i$, where $\alpha_1$, $\ldots$, $\alpha_m\in \mathcal{A}_X^{p}(U)$, then $\emph{supp}\alpha_i\subseteq\emph{supp}\omega\cap U$, for any $1\leq i\leq m$.

$(2)$ Assume that $X$ is oriented. For $S\in\Gamma(X, \mathcal{V}\otimes\mathcal{D}_X^{\prime p})$, if $S|_U=\sum_{i=1}^{m}e_i\otimes T_i$, where $T_1$, $\ldots$, $T_m\in \mathcal{D}_X^{\prime p}(U)$, then $\emph{supp}T_i\subseteq\emph{supp}S\cap U$, for any $1\leq i\leq m$.
\end{lem}

\begin{lem}\label{pull-push-support}
Let $f:Y\rightarrow X$ be a smooth map of connected smooth manifolds and $\mathcal{V}$ a local system of $\mathbb{R}$-modules of finite rank on $X$. Set $r=\emph{dim}Y-\emph{dim}X$.

$(1)$ For a $\mathcal{V}$-valued form $\omega$ on $X$,
\begin{displaymath}
\emph{supp}f^*\omega\subseteq f^{-1}(\emph{supp}\omega).
\end{displaymath}

$(2)$ Suppose that $X$ and $Y$ are oriented. If a $f^{-1}\mathcal{V}$-valued current $S$ on $Y$  satisfies that $f|_{\emph{supp}S}:\emph{supp}S\rightarrow X$ is proper, then
\begin{displaymath}
\emph{supp}f_*S\subseteq f(\emph{supp}S).
\end{displaymath}
\end{lem}

\begin{prop}\label{com2}
Let $f:Y\rightarrow X$ be a proper smooth map of connected smooth manifolds and $\mathcal{V}$ a local system of $\mathbb{R}$-modules of finite rank on $X$. For an open set $V$ of $X$, let $f_V:f^{-1}(V)\rightarrow V$ be the restriction of $f$ to $f^{-1}(V)$ and let  $j:V\rightarrow X$ and $j':f^{-1}(V)\rightarrow Y$  be inclusions. Then, $j'_*f_V^*=f^*j_*$ on $\Gamma_c(V,\mathcal{V}\otimes\mathcal{A}_X^{*})$. Moreover, if $X$ and $Y$ are oriented,  $f_{V*}j'^*=j^*f_*$ on $\Gamma(Y,f^{-1}\mathcal{V}\otimes\mathcal{D}_Y^{\prime*})$.
\end{prop}
Assume that $f:Y\rightarrow X$ is a smooth map between oriented smooth manifolds. Let $\mathcal{V}$ and $\mathcal{W}$ be local systems of $\mathbb{R}$-modules of finite ranks on $X$. For a $\mathcal{W}$-valued form $\omega$ on $X$ and a $f^{-1}\mathcal{V}$-valued current $T$ on $Y$ satisfying that $f|_{\textrm{supp}T}$ is proper, we have the \emph{projection formula}
\begin{equation}\label{pro-formula3}
f_*(T\wedge f^*\omega)=f_*T\wedge\omega.
\end{equation}
Hence
\begin{equation}\label{pro-formula4}
f_*(\varphi\cup f^*\eta)=f_*\varphi\cup\eta
\end{equation}
for $\varphi\in H_c^{*}(Y,f^{-1}\mathcal{V})$, $\eta\in H^{*}(X,\mathcal{W})$. Moreover, if $f$ is proper, they also hold for  $\varphi\in H^{*}(Y,f^{-1}\mathcal{V})$, $\eta\in H^{*}(X,\mathcal{W})$, or $\varphi\in H_c^{*}(Y,f^{-1}\mathcal{V})$, $\eta\in H_c^{*}(X,\mathcal{W})$, or $\varphi\in H^{*}(Y,f^{-1}\mathcal{V})$, $\eta\in H_c^{*}(X,\mathcal{W})$.

Suppose that $X$ is a connected oriented smooth manifold with dimension $n$ and $\mathcal{V}$ is a local system of $\mathbb{R}$-modules of rank $m$. Denote by $\mathcal{V}^{\vee}=\mathcal{H}om_{\underline{\mathbb{R}}_X}(\mathcal{V},\underline{\mathbb{R}}_X)$ the dual of $\mathcal{V}$.
Choose  a $\mathcal{V}$-constant covering $\mathfrak{U}$ of $X$. For $U\in\mathfrak{U}$, assume that $e^U_1$, $\ldots$, $e^U_m$ and $f^U_1$, $\ldots$, $f^U_m$ are bases of $\Gamma(U,\mathcal{V})$ and $\Gamma(U, \mathcal{V}^{\vee})$, respectively. For $\Omega\in\Gamma_c(X,\mathcal{V}\otimes\mathcal{V}^{\vee}\otimes\mathcal{A}_X^n)$, the restriction of $\Omega$ on $U$
\begin{displaymath}
\Omega|_U=\sum_{1\leq i\leq m}\sum_{1\leq j\leq m} e^U_i\otimes f^U_j\otimes\Omega_{ij},
\end{displaymath}
where $\Omega_{ij}$ is a smooth $n$-form on $U$ for $1\leq i,j\leq m$. Then
\begin{displaymath}
\sum_{1\leq i\leq m}\sum_{1\leq j\leq m} \langle e^U_i, f^U_j\rangle\Omega_{ij}
\end{displaymath}
is a smooth $n$-form on $U$, where $\langle,\rangle$ is the contraction between $\mathcal{V}$ and $\mathcal{V}^{\vee}$. Immediately, we construct a smooth $n$-form on $X$, denoted by $\textrm{tr}\Omega$, which does not depend on the choice of open coverings. Since
 $\textrm{supp}\textrm{(}\textrm{tr}\Omega\textrm{)}\subseteq \textrm{supp}\Omega$,  $\textrm{tr}$ gives a \emph{trace map}
\begin{displaymath}
\Gamma_c(X,\mathcal{V}\otimes\mathcal{V}^{\vee}\otimes\mathcal{A}_X^n)\rightarrow \Gamma_c(X,\mathcal{A}_X^n).
\end{displaymath}

\subsection{Cohomology with compact vertical supports}
Assume that $\pi:E\rightarrow X$ is a smooth fiber bundle on a smooth manifold $X$. Set
\begin{displaymath}
cv=\{Z\subseteq E|Z \mbox{ is closed in }E\mbox{ and }\pi|_Z:Z\rightarrow X\mbox{ is proper}\}.
\end{displaymath}
An element in $cv$ is called \emph{a compact vertical support}. $Z\in cv$, if and only if, $\pi^{-1}(K)\cap Z$ is compact for any compact subset $K\subseteq X$. The set $cv$ has following properties:

$(i)$ $cv$ is a paracompactifying family of supports on $E$ (\cite{Br}, IV 5.3 $(b)$, 5.5).

$(ii)$ $\mathcal{A}_E^p$ is $cv$-soft (\cite{Br}, II 9.4, 9.16) and hence  $\Gamma_{cv}$-acyclic (\cite{Br}, II 9.11).

$(iii)$ If $E$ is an oriented manifold, then $\mathcal{D}_E^{\prime p}$ is $cv$-soft (\cite{Br}, II 9.4, 9.16) and hence  $\Gamma_{cv}$-acyclic (\cite{Br}, II 9.11).

Let $\mathcal{W}$ be a local system of $\mathbb{R}$-modules of finite rank on $E$. Then $0\rightarrow\mathcal{W}\rightarrow \mathcal{W}\otimes\mathcal{A}_E^\bullet$ is a $cv$-soft resolution of $\mathcal{W}$ on $E$. By \cite{Br}, II, 4.1, the \emph{compact vertical cohomology} can be computed by
\begin{displaymath}
H_{cv}^*(E,\mathcal{W})\cong H^*(\Gamma_{cv}(E,\mathcal{W}\otimes\mathcal{A}_{E}^{\bullet})).
\end{displaymath}

Let $i_x:E_x\rightarrow E$ be the inclusion of the fiber $E_x$ over $x$ into $E$. Set $\omega\in \Gamma_{cv}(E,\mathcal{W}\otimes\mathcal{A}_{E}^*)$. By Lemma \ref{pull-push-support}, $\textrm{supp}(i_x^*\omega)\subseteq E_x\cap \textrm{supp}\omega$, hence is compact . The restrictions give a morphism
\begin{displaymath}
i_x^*:\Gamma_{cv}(E,\mathcal{W}\otimes\mathcal{A}_{E}^*)\rightarrow\Gamma_{c}(E_x,i_x^{-1}\mathcal{W}\otimes\mathcal{A}_{E_x}^*),
\end{displaymath}
which induces a morphism
\begin{displaymath}
H_{cv}^*(E,\mathcal{W})\rightarrow H_{c}^*(E_x,i_x^{-1}\mathcal{W}).
\end{displaymath}
If $E$ is an oriented manifold, the natural morphism $\mathcal{A}_E^\bullet\hookrightarrow\mathcal{D}_E^{\prime\bullet}$ of complexes of sheaves induces an isomorphism $H^*(\Gamma_{cv}(E,\mathcal{W}\otimes\mathcal{A}_{E}^{\bullet}))\cong H^*(\Gamma_{cv}(E,\mathcal{W}\otimes\mathcal{D}_{E}^{\prime\bullet}))$.

Let $X\times Y$ be viewed as a trivial smooth fiber bundle over $X$ and $\mathcal{W}$ a local system of $\mathbb{R}$-modules of finite rank on $X\times Y$.
\begin{lem}\label{cv-support}
Suppose that $V$ is an open set in $Y$ and $\omega$ is in $\mathcal{A}^*_{cv}(X\times V,\mathcal{W})$, where $X\times V$ is viewed as a smooth fiber bundle over $X$. Then $\emph{supp}\omega$ is closed in $X\times Y$.
\end{lem}
\begin{proof}
Suppose that  $\{B_\alpha\}$ is an open covering of $X$ such that $\overline{B}_\alpha$ are compact for all $\alpha$. Since $\omega$ has a compact vertical support, $\textrm{supp}\omega\cap(\overline{B}_\alpha\times V)$ is compact. Then
\begin{displaymath}
\begin{aligned}
B_\alpha\times Y-\textrm{supp}\omega=&B_\alpha\times Y-(B_\alpha\times V)\cap\textrm{supp}\omega\\
=&\left[\overline{B}_\alpha\times Y-(\overline{B}_\alpha\times V)\cap\textrm{supp}\omega\right]\cap(B_\alpha\times Y)
\end{aligned}
\end{displaymath}
is open in $B_\alpha\times Y$. So
\begin{displaymath}
X\times Y-\textrm{supp}\omega=\bigcup_\alpha(B_\alpha\times Y-\textrm{supp}\omega)
\end{displaymath}
is open in $X\times Y$. We complete the proof.
\end{proof}
For open subsets $V\subseteq U$ of $Y$, denote by $j_{VU}:X\times V\rightarrow X\times U$ the inclusion. By Lemma \ref{cv-support}, any $\omega\in\mathcal{A}^*_{cv}(X\times V,\mathcal{W})$ can be extended by zero on $(X\times V)\cup(X\times U-\textrm{supp}\omega)=X\times U$, denoted by $j_{VU*}\omega$. Since $\textrm{supp}(j_{VU*}\omega)=\textrm{supp}\omega$,  $j_{VU*}$ gives an operator
\begin{equation}\label{cv}
\mathcal{A}^*_{cv}(X\times V,\mathcal{W})\rightarrow\mathcal{A}^*_{cv}(X\times U,\mathcal{W}).
\end{equation}

Let $\pi:E\rightarrow X$ be an \emph{oriented} smooth \emph{fiber bundle} on a $($\emph{not necessarily orientable}$)$ smooth manifold $X$ and $r=\textrm{dim}E-\textrm{dim}X$. For $\omega\in \mathcal{A}_{cv}^p(E)$, denote by $\pi_*\omega$ the integral along fibers, seeing \cite{M1}, Sec. 2.4. 
Suppose that $\mathcal{V}$ is a local system of $\mathbb{R}$-modules of rank $m$ on $X$ and $\Omega\in \Gamma_{cv}(E, \pi^{-1}\mathcal{V}\otimes\mathcal{A}_E^p)$. Let $U$ be a $\mathcal{V}$-constant open subset of $X$ and $v_1$, $\ldots$, $v_m$ a basis of $\Gamma(U,\mathcal{V})$. By Lemma \ref{support2}, there exist $\alpha_1$, $\ldots$, $\alpha_m\in\mathcal{A}_{cv}^p(E_U)$ such that
\begin{displaymath}
\Omega|_{E_U}=\sum_{i=1}^m\pi_U^*v_i\otimes\alpha_i,
\end{displaymath}
where  $E_U=\pi^{-1}(U)$. For all such $U$, the $\mathcal{V}$-valued $(p-r)$-forms
\begin{displaymath}
\sum_{i=1}^mv_i\otimes\pi_{U*}\alpha_i
\end{displaymath}
can be glued as a global section on $X$, denoted by $\pi_*\Omega$. The definition here is similar to that of the pushout (seeing Section 3.1.3 (5)). Notice that the pushouts are defined on currents, hence it is necessary that the target manifolds are oriented. $\pi_*$ gives a morphism
\begin{displaymath}
\mathcal{A}_{cv}^*(E,\pi^{-1}\mathcal{V})\rightarrow\mathcal{A}^{*-r}(X,\mathcal{V}).
\end{displaymath}
Since $\textrm{supp}(\pi_*\Omega)\subseteq\pi_*(\textrm{supp}\Omega)$, the restriction of $\pi_*$ to $\mathcal{A}_{c}^*(E,\pi^{-1}\mathcal{V})\subseteq\mathcal{A}_{cv}^*(E,\pi^{-1}\mathcal{V})$ gives a morphism
\begin{displaymath}
\mathcal{A}_{c}^*(E,\pi^{-1}\mathcal{V})\rightarrow\mathcal{A}_c^{*-r}(X,\mathcal{V}).
\end{displaymath}
If $\mathcal{U}$ is a local system of $\mathbb{R}$-modules of finite rank on $X$ and $\omega\in \Gamma(X,\mathcal{U}\otimes\mathcal{A}_X^q)$, we have the projection formula
\begin{equation}\label{pro-formula5}
\pi_*(\Omega\wedge\pi^*\omega)=\pi_*\Omega\wedge\omega,
\end{equation}
which can be checked locally by  \cite{BT}, Prop. 6.15.

On the level of cohomology, $\pi_*$ induces
$H_{cv}^*(E,\pi^{-1}\mathcal{V})\rightarrow H^{*-r}(X,\mathcal{V})$ and $H_{c}^*(E,\pi^{-1}\mathcal{V})\rightarrow H_c^{*-r}(X,\mathcal{V})$,
since $\pi_*\textrm{d}=\textrm{d}\pi_*$.

Moreover, assume that $E$ is an \emph{oriented} smooth \emph{vector bundle} over an \emph{oriented} smooth manifold $X$. Let $i:X\rightarrow E$ be the inclusion of the zero section. For $T\in \Gamma(X, i^{-1}\mathcal{W}\otimes\mathcal{D}_X^{\prime*})$, $i_*T\in \Gamma_{cv}(E,\mathcal{W}\otimes\mathcal{D}_E^{\prime*+r})$ by Lemma \ref{pull-push-support}. So $i_*$ induce a morphism $i_*:H^*(X, i^{-1}\mathcal{W})\rightarrow H_{cv}^{*+r}(E,\mathcal{W})$.

\subsection{Generalizations of R. O. Wells' results}
In \cite{W}, R. O. Wells  compared de Rham and Dolbeault cohomology for proper surjective maps. We simply extend his results to more general cases.

\cite{W}, Thm. 3.3 was generalized  as follows.
\begin{prop}\label{i-s}
\textsc{}Let $f:Y\rightarrow X$ be a proper surjective smooth map of connected oriented  smooth manifolds with the same dimensions and $\emph{deg}f\neq 0$. If $\mathcal{V}$ is a local system of $\mathbb{R}$ or $\mathbb{C}$-modules of finite rank on $X$, then, for any $p$, $f^*:H^p(X,\mathcal{V})\rightarrow H^p(Y,f^{-1}\mathcal{V})$ is injective and $f_*:H^{p}(Y,f^{-1}\mathcal{V})\rightarrow H^{p}(X,\mathcal{V})$ is surjective. Moreover, they also hold for the cohomologies with compact supports.
\end{prop}
\begin{proof}
Notice that $f_*1_{Y}=\textrm{deg} f\cdot1_X\neq0$ in $H^0(X,\mathbb{R})$, where $1_X$ and $1_{Y}$ are classes of the constant $1$ in $H^0(X,\mathbb{R})$ and $H^0(Y,\mathbb{R})$, respectively.
By the projection formula (\ref{pro-formula4}), $f_*f^*\eta=\textrm{deg} f\cdot \eta$, where $\eta\in H^{p}(X,\mathcal{V})$ or $H_{c}^{p}(X,\mathcal{V})$. We get the proposition immediately.
\end{proof}

Recall that a complex manifold $X$ is called \emph{$p$-K\"ahlerian}, if it admits a closed strictly positive $(p,p)$-form $\Omega$ $($\cite{AB}, Def. 1.1, 1.2$)$. In such case, $\Omega|_{Z}$ is a volume form on $Z$, for any complex submanifold $Z$ of pure dimension $p$ of $X$. Any complex manifold is $0$-K\"ahlerian and any K\"ahler manifold $X$ is  $p$-K\"ahlerian for every $p\leq \textrm{dim}_{\mathbb{C}}X$. We generalize  \cite{W}, Thm. 3.1 and 4.1 as follows.

\begin{prop}\label{inj-surj}
Suppose that $f:Y\rightarrow X$ is a proper surjective holomorphic map between connected complex manifolds and $Y$ is $r$-K\"ahlerian, where $r=\emph{dim}_{\mathbb{C}}Y-\emph{dim}_{\mathbb{C}}X$. Let $\mathcal{V}$ be a local system of $\mathbb{R}$ or $\mathbb{C}$-modules of finite rank and  $\mathcal{E}$ a locally free sheaf of  $\mathcal{O}_X$-modules of finite rank on $X$, respectively.  Then, for any $p$, $q$,

$(1)$ $f^*:H^p(X,\mathcal{V})\rightarrow H^p(Y,f^{-1}\mathcal{V})$ and $f^*:H^{p,q}(X,\mathcal{E})\rightarrow H^{p,q}(Y,f^*\mathcal{E})$ are injective,

$(2)$ $f_*:H^{p}(Y,f^{-1}\mathcal{V})\rightarrow H^{p-2r}(X,\mathcal{V})$ and $f_*:H^{p,q}(Y,f^*\mathcal{E})\rightarrow H^{p-r,q-r}(X,\mathcal{E})$ are surjective.

Moreover, they also hold for the cohomologies with compact supports.
\end{prop}
\begin{proof}
Let $\Omega$ be a strictly positive closed $(r,r)$-form on $Y$.  Then $c=f_*\Omega$ is a closed  current of degree $0$, hence a constant. By Sard's theorem, the set $X_0$ of regular values of $f$ is nonempty. For any $x\in X_0$, $Y_x=f^{-1}(x)$ is a compact complex submanifold of pure dimension $r$, so $c=\int_{Y_x}\Omega|_{Y_x}>0$. By the projection formula (\ref{pro-formula4}) (resp. (\ref{pro-formula2})), $f_*([\Omega]\cup f^*\eta)=c\cdot \eta$, where $[\Omega]\in H^{2r}(Y,\mathbb{R})$ or $H^{2r}(Y,\mathbb{C})$ (resp. $H^{r,r}(Y)$) and $\eta\in H^{p}(X,\mathcal{V})$ (resp. $H^{p,q}(X,\mathcal{E})$) or $H_{c}^{p}(X,\mathcal{V})$ (resp. $H_c^{p,q}(X,\mathcal{E})$). It is easily to deduce the proposition.
\end{proof}

\section{Examples}
\subsection{$\mathcal{H}_X^{p,*}(\mathcal{E})$ and $\mathcal{H}_{X,c}^{p,*}(\mathcal{E})$}
Let $X$ be a connected complex manifold and $\mathcal{E}$  a locally free sheaf of $\mathcal{O}_X$-modules of finite rank on $X$. For open sets $V\subseteq U$, set $j_{VU}:V\rightarrow U$ the inclusion. Define $\mathcal{H}_X^{p,q}(\mathcal{E})$ as
\begin{displaymath}
\mathcal{H}_X^{p,q}(\mathcal{E})(U)=H^{p,q}(U,\mathcal{E})
\end{displaymath}
 for any open set $U$ in $X$, and the restriction
\begin{displaymath}
\rho_{U,V}=j_{VU}^*:\mathcal{H}_X^{p,q}(\mathcal{E})(U)\rightarrow\mathcal{H}_X^{p,q}(\mathcal{E})(V)
\end{displaymath}
for open sets $V\subseteq U$. Define  $\mathcal{H}_{X,c}^{p,q}(\mathcal{E})$ as
\begin{displaymath}
\mathcal{H}_{X,c}^{p,q}(\mathcal{E})(U)=H_c^{p,q}(U,\mathcal{E})
\end{displaymath}
 for any open set $U$ in $X$, and the extension
\begin{displaymath}
i_{V,U}=j_{VU*}:\mathcal{H}_{X,c}^{p,q}(\mathcal{E})(V)\rightarrow\mathcal{H}_{X,c}^{p,q}(\mathcal{E})(U)
\end{displaymath}
for open sets $V\subseteq U$.  It is easily to check that $\mathcal{H}_X^{p,q}(\mathcal{E})$ and $\mathcal{H}_{X,c}^{p,q}(\mathcal{E})$ are a cdp presheaf  and a cds precosheaf of $\mathbb{C}$-modules on $X$, respectively.

Let $\mathcal{F}$ be a locally free sheaf of $\mathcal{O}_X$-modules of finite rank on $X$ and $\Omega$ an element in $H^{s,t}(X,\mathcal{F})$. Assume that the $\mathcal{E}$-valued $(s,t)$-form $\omega$ is a representative of $\Omega$. For convenience, denote $\mathcal{F}^{p,q}=\mathcal{E}\otimes\mathcal{A}_X^{p,q}$ and $\mathcal{G}^{p,q}=\mathcal{E}\otimes\mathcal{F}\otimes\mathcal{A}_X^{p+s,q+t}$, for any $p$, $q$. Given $p$, for open subsets $U$ and $V$  of $X$, there exist two exact sequences of complexes
\begin{displaymath}
\small{\xymatrix{
   0\ar[r] &\Gamma(U\cup V,\mathcal{F}^{p,\bullet})\ar[d]^{\omega|_{U\cup V}\wedge} \ar[r]^{(j_1^*,j_2^*)\quad\quad} &\Gamma(U,\mathcal{F}^{p,\bullet})\oplus \Gamma(V,\mathcal{F}^{p,\bullet})\ar[d]^{(\omega|_{U}\wedge,\omega|_{V}\wedge)}\ar[r]^{\quad\quad j_3^*-j_4^*}&\Gamma(U\cap V,\mathcal{F}^{p,\bullet})\ar[d]^{\omega|_{U\cap V}\wedge}\ar[r] &0\\
 0\ar[r] & \Gamma(U\cup V,\mathcal{G}^{p,\bullet})     \ar[r]^{(j_1^*,j_2^*)\quad\quad} & \Gamma(U,\mathcal{G}^{p,\bullet})\oplus \Gamma(V,\mathcal{G}^{p,\bullet})\ar[r]^{\quad\quad j_3^*-j_4^*}&\Gamma(U\cap V,\mathcal{G}^{p,\bullet})\ar[r] &0 }}
\end{displaymath}
and
\begin{displaymath}
\small{\xymatrix{
   0\ar[r] &\Gamma_c(U\cap V,\mathcal{F}^{p,\bullet})\ar[d]^{\omega|_{U\cap V}\wedge} \ar[r]^{(j_{3*},j_{4*})\quad\quad} &\Gamma_c(U,\mathcal{F}^{p,\bullet})\oplus \Gamma_c(V,\mathcal{F}^{p,\bullet})\ar[d]^{(\omega|_{U}\wedge,\omega|_{V}\wedge)}\ar[r]^{\quad\quad j_{1*}-j_{2*}}&\Gamma_c(U\cup V,\mathcal{F}^{p,\bullet})\ar[d]^{\omega|_{U\cup V}\wedge}\ar[r] &0\\
 0\ar[r] & \Gamma_c(U\cap V,\mathcal{G}^{p,\bullet})     \ar[r]^{(j_{3*},j_{4*})\quad\quad} & \Gamma_c(U,\mathcal{G}^{p,\bullet})\oplus \Gamma_c(V,\mathcal{G}^{p,\bullet})\ar[r]^{\quad\quad j_{1*}-j_{2*}}&\Gamma_c(U\cup V,\mathcal{G}^{p,\bullet})\ar[r] &0,}}
\end{displaymath}
where $j_i$ are self-explanatory inclusions and $j_i^*$, $j_{i*}$ are restrictions, extensions by zero, respectively, for $i=1,2,3,4$. For the exactness, we refer to \cite{BT}, Prop. 2.3, 2.7. They induce two commutative diagrams of long exact sequences on the level of cohomology
\begin{displaymath}
\tiny{\xymatrix{
   \ar[r]^{\delta^{p,q-1}_{\mathcal{E}}\quad\qquad} & H^{p,q}(U\cup V,\mathcal{E})\ar[d]^{\Omega|_{U\cup V}\cup} \ar[r]^{(j_1^*,j_2^*)\quad\quad} &H^{p,q}(U,\mathcal{E})\oplus H^{p,q}(V,\mathcal{E})\ar[d]^{(\Omega|_{U}\cup,\Omega|_{V}\cup)}\ar[r]^{\quad\quad j_3^*-j_4^*}&H^{p,q}(U\cap V,\mathcal{E})\ar[d]^{\Omega|_{U\cup V}\cup}\ar[r]^{\quad\qquad\delta^{p,q}_{\mathcal{E}}} &\cdots\\
  \ar[r]^{\delta^{p+s,q+t-1}_{\mathcal{E}\otimes\mathcal{F}}\quad\qquad\qquad} & H^{p+s,q+t}(U\cup V,\mathcal{E}\otimes\mathcal{F})\ar[r]^{(j_1^*,j_2^*)\qquad\qquad} &H^{p+s,q+t}(U,\mathcal{E}\otimes\mathcal{F})\oplus H^{p+s,q+t}(V,\mathcal{E}\otimes\mathcal{F})\ar[r]^{\qquad\qquad j_3^*-j_4^*}&H^{p+s,q+t}(U\cap V,\mathcal{E}\otimes\mathcal{F})\ar[r]^{\quad\qquad\qquad\delta^{p+s,q+t}_{\mathcal{E}\otimes\mathcal{F}}} &\cdots}}
\end{displaymath}
and
\begin{displaymath}
\tiny{\xymatrix{
   \ar[r]^{\delta^{p,q-1}_{\mathcal{E},c}\quad\qquad} & H_c^{p,q}(U\cap V,\mathcal{E})\ar[d]^{\Omega|_{U\cap V}\cup} \ar[r]^{(j_{3*},j_{4*})\quad\quad} &H_c^{p,q}(U,\mathcal{E})\oplus H_c^{p,q}(V,\mathcal{E})\ar[d]^{(\Omega|_{U}\cup,\Omega|_{V}\cup)}\ar[r]^{\quad\quad j_{1*}-j_{2*}}&H_c^{p,q}(U\cup V,\mathcal{E})\ar[d]^{\Omega|_{U\cup V}\cup}\ar[r]^{\quad\qquad\delta^{p,q}_{\mathcal{E},c}} &\cdots\\
  \ar[r]^{\delta^{p+s,q+t-1}_{\mathcal{E}\otimes\mathcal{F},c}\quad\qquad\qquad} & H_c^{p+s,q+t}(U\cap V,\mathcal{E}\otimes\mathcal{F})\ar[r]^{(j_{3*},j_{4*})\qquad\qquad} &H_c^{p+s,q+t}(U,\mathcal{E}\otimes\mathcal{F})\oplus H_c^{p+s,q+t}(V,\mathcal{E}\otimes\mathcal{F})\ar[r]^{\qquad\qquad\quad  j_{1*}-j_{2*}}&H_c^{p+s,q+t}(U\cup V,\mathcal{E}\otimes\mathcal{F})\ar[r]^{\quad\qquad\qquad\delta^{p+s,q+t}_{\mathcal{E}\otimes\mathcal{F},c}} &\cdots.}}
\end{displaymath}
Define
\begin{displaymath}
\mathcal{H}_X^{p,*}(\mathcal{E})=\{(\mathcal{H}_X^{p,q}(\mathcal{E}), \mbox{ }\delta^{p,q}_{\mathcal{E}})|q\in\mathbb{Z}\}\mbox{ and } \mathcal{H}_{X,c}^{p,*}(\mathcal{E})=\{(\mathcal{H}_{X,c}^{p,q}(\mathcal{E}),\mbox{ } \delta^{p,q}_{\mathcal{E},c})|q\in\mathbb{Z}\}.
\end{displaymath}
Then, via cup products, $\Omega$ gives M-V morphisms $\Omega\cup\bullet:\mathcal{H}_X^{p,*}(\mathcal{E})\rightarrow \mathcal{H}_X^{p+s,*+t}(\mathcal{E}\otimes\mathcal{F})$ and $\Omega\cup\bullet:\mathcal{H}_{X,c}^{p,*}(\mathcal{E})\rightarrow \mathcal{H}_{X,c}^{p+s,*+t}(\mathcal{E}\otimes\mathcal{F})$.

Let $f:Y\rightarrow X$ be a holomorphic map of connected complex manifolds and $r=\textrm{dim}_{\mathbb{C}}Y-\textrm{dim}_{\mathbb{C}}X$. For any open subset $U$ in $X$, denote  $\widetilde{U}=f^{-1}(U)$. 
By the similar way to above cases, we get two commutative diagrams of long exact sequences
\begin{displaymath}
\tiny{\xymatrix{
  \cdots H^{p,q}(U\cup V,\mathcal{E})\ar[d]^{f^*_{U\cup V}} \ar[r]^{(j_1^*,j_2^*)\quad\quad} &H^{p,q}(U,\mathcal{E})\oplus H^{p,q}(V,\mathcal{E})\ar[d]^{(f^*_{U},f^*_{V})}\ar[r]^{\quad\quad j_3^*-j_4^*}&H^{p,q}(U\cap V,\mathcal{E})\ar[d]^{f^*_{U\cap V}}\ar[r]^{\delta^{p,q}_{\mathcal{E}}\quad} &H^{p,q+1}(U\cup V,\mathcal{E})\ar[d]^{f^*_{U\cup V}}\cdots\\
 \cdots H^{p,q}(\widetilde{U}\cup \widetilde{V},f^*\mathcal{E})     \ar[r]^{(\tilde{j}_1^*,\tilde{j}_2^*)\quad\quad} & H^{p,q}(\widetilde{U},f^*\mathcal{E})\oplus H^{p,q}(\widetilde{V},,f^*\mathcal{E})\ar[r]^{\quad\qquad \tilde{j}_3^*-\tilde{j}_4^*}&H^{p,q}(\widetilde{U}\cap \widetilde{V},,f^*\mathcal{E})\ar[r]^{\delta^{p,q}_{f^*\mathcal{E}}\quad} &H^{p,q+1}(\widetilde{U}\cup \widetilde{V},f^*\mathcal{E})\cdots}}
\end{displaymath}
and
\begin{displaymath}
\tiny{\xymatrix{
   \cdots H_c^{p,q}(\widetilde{U}\cap \widetilde{V},f^*\mathcal{E})\ar[d]^{f_{U\cap V*}} \ar[r]^{(\tilde{j}_{3*},\tilde{j}_{4*})\quad\quad} &H_c^{p,q}(\widetilde{U},f^*\mathcal{E})\oplus H_c^{p,q}(\widetilde{V},\mathcal{E})\ar[d]^{(f_{U*},f_{V*})}\ar[r]^{\quad\quad \tilde{j}_{1*}-\tilde{j}_{2*}}&H_c^{p,q}(\widetilde{U}\cup \widetilde{V},f^*\mathcal{E})\ar[d]^{f_{U\cup V*}}\ar[r]^{\delta^{p,q}_{f^*\mathcal{E},c}\quad} &H_c^{p,q+1}(\widetilde{U}\cap \widetilde{V},f^*\mathcal{E})\ar[d]^{f_{U\cap V*}}\cdots\\
 \cdots H_c^{p-r,q-r}(U\cap V,\mathcal{E})     \ar[r]^{(j_{3*},j_{4*})\quad\quad} & H_c^{p-r,q-r}(U,\mathcal{E})\oplus H_c^{p-r,q-r}(V,\mathcal{E})\ar[r]^{\quad\qquad \tilde{j}_{1*}-\tilde{j}_{2*}}&H_c^{p-r,q-r}(U\cup V,\mathcal{E})\ar[r]^{\delta^{p-r,q-r}_{\mathcal{E},c}\qquad} &H_c^{p-r,q-r+1}(U\cap V,\mathcal{E})\cdots,}}
\end{displaymath}
where $j_i$, $\tilde{j}_i$, $j_i^*$, $\tilde{j}_i^*$, $j_{i*}$, $\tilde{j}_{i*}$ are self-explanatory as above cases, for $i=1,2,3,4$. Hence, pullbacks and pushouts define  M-V morphisms $f^*:\mathcal{H}_X^{p,*}(\mathcal{E})\rightarrow f_*\mathcal{H}_Y^{p,*}(f^*\mathcal{E})$ and $f_*:f_*\mathcal{H}_{Y,c}^{p,*}(f^*\mathcal{E})\rightarrow \mathcal{H}_{X,c}^{p-r,*-r}(\mathcal{E})$ respectively.

Assume that $f$ is proper. It is similar to  prove that $f^*:\mathcal{H}_{X,c}^{p,*}(\mathcal{E})\rightarrow f_*\mathcal{H}_{Y,c}^{p,*}(f^*\mathcal{E})$  and $f_*:f_*\mathcal{H}_{Y}^{p,*}(f^*\mathcal{E})\rightarrow \mathcal{H}_{X}^{p-r,*-r}(\mathcal{E})$ are  M-V morphisms, where we use Proposition \ref{com1} to check the commutative diagrams on the level of forms and currents.

We summarize above results as follows.
\begin{prop}\label{example1}
Let $X$ be a connected complex manifold and $\mathcal{E}$  a locally free sheaf of $\mathcal{O}_X$-modules of finite rank on $X$. Fixed $p\in \mathbb{Z}$.

$(1)$ $\mathcal{H}_X^{p,*}(\mathcal{E})$ \emph{(resp.} $\mathcal{H}_{X,c}^{p,*}(\mathcal{E})$\emph{)} is a M-V system  of cdp presheaves \emph{(resp.} cds precosheaves\emph{)} of $\mathbb{C}$-modules on $X$.

$(2)$ Assume that $\mathcal{F}$ is a locally free sheaf of $\mathcal{O}_X$-modules of finite rank on $X$ and $\Omega$ is an element in $H^{s,t}(X,\mathcal{F})$. Then
\begin{displaymath}
\Omega\cup\bullet:\mathcal{H}_X^{p,*}(\mathcal{E})\rightarrow \mathcal{H}_X^{p+s,*}(\mathcal{E}\otimes\mathcal{F})[t]
\end{displaymath}
and
\begin{displaymath}
\Omega\cup\bullet:\mathcal{H}_{X,c}^{p,*}(\mathcal{E})\rightarrow \mathcal{H}_{X,c}^{p+s,*}(\mathcal{E}\otimes\mathcal{F})[t]
\end{displaymath}
are M-V morphisms.

$(3)$ Assume that $f:Y\rightarrow X$ is a holomorphic map of connected complex manifolds and $r=\emph{dim}_{\mathbb{C}}Y-\emph{dim}_{\mathbb{C}}X$. Then
\begin{displaymath}
f^*:\mathcal{H}_X^{p,*}(\mathcal{E})\rightarrow f_*\mathcal{H}_Y^{p,*}(f^*\mathcal{E})
\end{displaymath}
and
\begin{displaymath}
f_*:f_*\mathcal{H}_{Y,c}^{p,*}(f^*\mathcal{E})\rightarrow \mathcal{H}_{X,c}^{p-r,*}(\mathcal{E})[-r]
\end{displaymath}
are M-V morphisms.

Moreover, if $f$ is proper, then
\begin{displaymath}
f^*:\mathcal{H}_{X,c}^{p,*}(\mathcal{E})\rightarrow f_*\mathcal{H}_{Y,c}^{p,*}(f^*\mathcal{E})
\end{displaymath}
and
\begin{displaymath}
f_*:f_*\mathcal{H}_{Y}^{p,*}(f^*\mathcal{E})\rightarrow \mathcal{H}_{X}^{p-r,*}(\mathcal{E})[-r]
\end{displaymath}
are M-V morphisms.
\end{prop}

\subsection{$\mathcal{H}_X^*(\mathcal{V})$ and $\mathcal{H}_{X,c}^*(\mathcal{V})$}
Following Section 4.1, we can define $\mathcal{H}_X^*(\mathcal{V})$ and $\mathcal{H}_{X,c}^*(\mathcal{V})$ for any local system $\mathcal{V}$ of $\mathbb{R}$-modules of finite rank on a smooth manifold $X$. Analogue to  Propositon \ref{example1}, we get the following proposition.
\begin{prop}\label{example2}
Let $X$ be a connected smooth manifold and $\mathcal{V}$  a local system of $\mathbb{R}$-modules of finite rank on $X$.

$(1)$  $\mathcal{H}_X^*(\mathcal{V})$ \emph{(resp.} $\mathcal{H}_{X,c}^{p}(\mathcal{V})$\emph{)} is a M-V system of cdp presheaves \emph{(resp.} cds precosheaves\emph{)}.

$(2)$ Assume that $\mathcal{W}$ is a local system of $\mathbb{R}$-modules of finite rank on $X$ and $\Omega$ is an element in $H^s(X,\mathcal{W})$. Then
\begin{displaymath}
\Omega\cup\bullet:\mathcal{H}_X^{*}(\mathcal{V})\rightarrow \mathcal{H}_X^{*}(\mathcal{V}\otimes\mathcal{W})[s]
\end{displaymath}
and
\begin{displaymath}
\Omega\cup\bullet:\mathcal{H}_{X,c}^{*}(\mathcal{V})\rightarrow \mathcal{H}_{X,c}^{*}(\mathcal{V}\otimes\mathcal{W})[s]
\end{displaymath}
are M-V morphisms.

$(3)$ Assume that $f:Y\rightarrow X$ is a smooth map of connected smooth manifolds and $r=\emph{dim}Y-\emph{dim}X$.

$(i)$ $f^*:\mathcal{H}_X^{*}(\mathcal{V})\rightarrow f_*\mathcal{H}_Y^{*}(f^{-1}\mathcal{V})$ is a  M-V morphism. Moreover, if $f$ is proper, then $f^*:\mathcal{H}_{X,c}^{p,*}(\mathcal{V})\rightarrow f_*\mathcal{H}_{Y,c}^{p,*}(f^{-1}\mathcal{V})$   is a  M-V morphism.

$(ii)$ Suppose that $X$ and $Y$ are oriented. $f_*:f_*\mathcal{H}_{Y,c}^{*}(f^{-1}\mathcal{V})\rightarrow \mathcal{H}_{X,c}^{*}(\mathcal{V})[-r]$  is a M-V morphism. Moreover, if $f$ is proper, then $f_*:f_*\mathcal{H}_{Y}^{*}(f^{-1}{\mathcal{V}})\rightarrow \mathcal{H}_{X}^{*}(\mathcal{V})[-r]$ is a M-V morphism.
\end{prop}
\subsection{$\mathcal{H}_{E,cv}^{*}(\mathcal{W})$ and $^\prime\mathcal{H}_{X\times Y,cv}^{*}(\mathcal{W})$}
Let $\pi:E\rightarrow X$ be a smooth fiber bundle on a smooth manifold $X$ and  $\mathcal{W}$ a local system of $\mathbb{R}$-modules of finite rank on $E$. Define  $\mathcal{H}_{E,cv}^{p}(\mathcal{W})$ as follows:
\begin{displaymath}
\mathcal{H}_{E,cv}^{p}(\mathcal{W})(U)=H_{cv}^{p}(E_U,\mathcal{W})
\end{displaymath}
 for any open set $U$ in $X$, and the restriction
\begin{displaymath}
\rho_{U,V}=\tilde{j}_{VU}^*:\mathcal{H}_{E,cv}^{p}(\mathcal{W})(U)\rightarrow\mathcal{H}_{E,cv}^{p}(\mathcal{W})(V)
\end{displaymath}
for any open sets $V\subseteq U$ in $X$, where  $\tilde{j}_{VU}:E_V\rightarrow E_U$ is the inclusion. Denote $\mathcal{F}^p=\mathcal{W}\otimes\mathcal{A}_E^p$. Using a partition of unit subordinate to $\{U,V\}$, we get a short exact sequence of complexes (seeing \cite{BT}, Prop. 2.3)
\begin{equation}\label{cv-exact1}
\small{\xymatrix{
 0\ar[r] & \Gamma_{cv}(E_{U\cup V},\mathcal{F}^\bullet)   \ar[r]^{P_{U,V}\qquad\quad} & \Gamma_{cv}(E_{U},\mathcal{F}^\bullet)\oplus\Gamma_{cv}(E_V,\mathcal{F}^\bullet) \ar[r]^{\qquad\quad Q_{U,V}}&\Gamma_{cv}(E_{U\cap V},\mathcal{F}^\bullet)\ar[r] &0 }}
\end{equation}
for any open set $U$, $V$ in $X$, where $P_{U,V}=(\rho_{V\cup U,U},\mbox{ }\rho_{V\cup U,V})$ and $Q_{U,V}=\rho_{U,V\cap U}-\rho_{V,V\cap U}$.

View $X\times Y$ as a  smooth fiber bundle over $X$ and suppose that $\mathcal{W}$ is a local system of $\mathbb{R}$-modules of finite rank on $X\times Y$. Define  $^\prime\mathcal{H}_{X\times Y,cv}^{p}(\mathcal{W})$ as follows:
\begin{displaymath}
^\prime\mathcal{H}_{X\times Y,cv}^{p}(\mathcal{W})(U)=H_{cv}^{p}(X\times U,\mathcal{W})
\end{displaymath}
 for any open set $U$ in $Y$, and the extension
\begin{displaymath}
i_{V,U}=j'_{VU*}:^\prime\mathcal{H}_{X\times Y,cv}^{p}(\mathcal{W})(V)\rightarrow^\prime\mathcal{H}_{X\times Y,cv}^{p}(\mathcal{W})(U)
\end{displaymath}
for any open sets $V\subseteq U$ in $Y$, where $j'_{VU}:X\times V\rightarrow X\times U$ is the inclusion.  Denote $\mathcal{G}^p=\mathcal{W}\otimes\mathcal{A}_{X\times Y}^p$. For any open set $U$, $V$ in $Y$, there exists a short exact sequence of complexes  (seeing \cite{BT}, Prop. 2.7)
\begin{equation}\label{cv-exact2}
\tiny{\xymatrix{
 0\ar[r] & \Gamma_{cv}(X\times (U\cap V), \mathcal{G}^\bullet) \ar[r]^{P_{U,V}\quad\quad} & \Gamma_{cv}(X\times U, \mathcal{G}^\bullet)\oplus\Gamma_{cv}(X\times V, \mathcal{G}^\bullet) \ar[r]^{\quad\qquad Q_{U,V}}&\Gamma_{cv}(X\times (U\cup V), \mathcal{G}^\bullet)\ar[r] &0 }}
\end{equation}
by a partition of unit subordinate to $\{U,V\}$, , where $P_{U,V}=(i_{V\cap U,U},i_{V\cap U,V})$ and $Q_{U,V}=i_{U,V\cup U}-i_{V,V\cup U}$.

By (\ref{cv-exact1}) and (\ref{cv-exact2}), we easily get following propositions.
\begin{prop}\label{example3}
Let $\pi:E\rightarrow X$ be a smooth fiber bundle on a connected smooth manifold $X$ and $\mathcal{V}$,  $\mathcal{W}$ local systems of $\mathbb{R}$-modules of finite ranks on $X$, $E$ respectively.

$(1)$ $\mathcal{H}_{E,cv}^*(\mathcal{W})$ is a M-V system of cdp presheaves on $X$.

$(2)$ Assume that $\mathcal{U}$ is a local system of $\mathbb{R}$-modules of finite rank on $E$ and  $\Omega$ is an element in $H^s(E,\mathcal{U})$ \emph{(resp.} $H_{cv}^s(E,\mathcal{U})$  \emph{)}. Then
\begin{displaymath}
\Omega\cup\bullet:\mathcal{H}_{E,cv}^*(\mathcal{W})\rightarrow \mathcal{H}_{E,cv}^{*}(\mathcal{U}\otimes\mathcal{W})[s]
\end{displaymath}
\emph{(resp.}
\begin{displaymath}
\Omega\cup\bullet:\mathcal{H}_{E}^*(\mathcal{W})\rightarrow \mathcal{H}_{E,cv}^{*}(\mathcal{U}\otimes\mathcal{W})[s]\emph{)}
\end{displaymath}
is a M-V morphism.

$(3)$ For any $\Omega\in H_{cv}^s(E,\mathcal{U})$, $\Omega\cup\pi^*(\bullet):\mathcal{H}_{X,c}^*(\mathcal{V})\rightarrow \mathcal{H}_{E,c}^{*+s}(\mathcal{U}\otimes\pi^{-1}\mathcal{V})[s]$
is a M-V morphism.

$(4)$ Set $r=\emph{dim}E-\emph{dim}X$. If $E$ is an oriented fiber bundle, then $\pi_*:\mathcal{H}_{E,cv}^*(\pi^{-1}\mathcal{V})\rightarrow\mathcal{H}_{X}^{*}(\mathcal{V})[-r]$ is a M-V morphism. In addition, if $X$ is  oriented and $E$ is a vector bundle, then $i_*:\mathcal{H}_{X}^*(i^{-1}\mathcal{W})\rightarrow\mathcal{H}_{E,cv}^{*}(\mathcal{W})[r]$ is a  M-V morphism.
\end{prop}
\begin{prop}\label{example4}
Suppose that $X$ and $Y$ are connected smooth manifolds. Let $\mathcal{V}$ and  $\mathcal{W}$ be local systems of $\mathbb{R}$-modules of finite ranks on $Y$ and $X\times Y$, respectively. If we  view $X\times Y$ as a smooth fiber bundle over $X$, then

$(1)$ $^\prime\mathcal{H}_{X\times Y,cv}^*(\mathcal{W})$  is a M-V system of cds precosheaves  on $Y$,

$(2)$ $\emph{pr}_2^*:\mathcal{H}_{Y,c}^*(\mathcal{V})\rightarrow^\prime\mathcal{H}_{X\times Y,cv}^*(\emph{pr}_2^{-1}\mathcal{V})$ is a M-V morphism, where $\emph{pr}_2:X\times Y\rightarrow Y$ is the second projection.
\end{prop}

\section{Some applications}
Let $X$ be a smooth manifold (resp. complex manifold) and $\mathcal{V}$ (resp. $\mathcal{E}$) a local system (resp. locally free sheaf) of $\mathbb{R}$ (resp. $\mathcal{O}_X$)- modules of finite rank on $X$. An open covering of $X$ is called \emph{a} \emph{$\mathcal{V}$-constant} (resp. \emph{$\mathcal{E}$-free}) \emph{basis}, if it is both a basis of topology  and a $\mathcal{V}$-constant (resp. $\mathcal{E}$-free) covering of $X$.

\emph{Notice}. If we use $\mathbb{C}$ instead of $\mathbb{R}$, all conclusions in this section still hold, except Theorem \ref{L-H} $(2)$, Corollary \ref{proj-bun} (2) and Proposition \ref{prop} (3)-(6).

\subsection{Poincar\'{e} dulaity theorem}
Denote by $M^\vee=\textrm{Hom}_{\mathbb{R}}(M,\mathbb{R})$ the dual space of a $\mathbb{R}$-vector space $M$ and denote by $\rho^\vee:N^\vee\rightarrow M^\vee$ the dual of a linear map $\rho:M\rightarrow N$.

Let $X$ be a connected oriented smooth manifold with dimension $n$ and $\mathcal{V}$ a local system of $\mathbb{R}$-modules of finite rank on $X$. For any $p$, set $\mathcal{F}^{p}=\mathcal{V}\otimes\mathcal{A}_X^{p}$ and $\mathcal{G}^{p}=\mathcal{V}^\vee\otimes\mathcal{A}_X^{n-p}$. For an open set $U$, define
\begin{displaymath}
PD_U(\alpha)(\beta)=\int_U \textrm{tr}(\alpha\wedge\beta)
\end{displaymath}
for $\alpha\in \Gamma(U,\mathcal{F}^p)$ and $\beta\in \Gamma_c(U,\mathcal{G}^p)$, which gives a linear map $PD_U:\Gamma(U,\mathcal{F}^p)\rightarrow\Gamma_c(U,\mathcal{G}^p)^\vee$.
For open subsets $U$, $V$  of $X$ and $\varphi\in\Gamma_c(U,\mathcal{G}^p)^\vee$, $\psi\in\Gamma_c(V,\mathcal{G}^p)^\vee$, set
\begin{displaymath}
I^p(\varphi,\psi)=\varphi\circ \textrm{pr}_1-\psi\circ \textrm{pr}_2:\Gamma_c(U,\mathcal{G}^p)\oplus \Gamma_c(V,\mathcal{G}^p)\rightarrow\mathbb{R},
\end{displaymath}
where  $\textrm{pr}_1$ and $\textrm{pr}_2$ are projections from $\Gamma_c(U,\mathcal{G}^{p})\oplus \Gamma_c(V,\mathcal{G}^{p})$ onto $\Gamma_c(U,\mathcal{G}^{p})$ and $\Gamma_c(V,\mathcal{G}^{p})$ respectively. $I^p$ gives an isomorphism
\begin{displaymath}
\Gamma_c(U,\mathcal{G}^p)^\vee\oplus \Gamma_c(V,\mathcal{G}^p)^\vee\rightarrow\left(\Gamma_c(U,\mathcal{G}^p)\oplus \Gamma_c(V,\mathcal{G}^p)\right)^\vee,
\end{displaymath}
There is a commutative diagram
\begin{displaymath}
\tiny{\xymatrix{
   0\ar[r] &\Gamma(U\cup V,\mathcal{F}^{\bullet})\ar[d]^{PD_{U\cup V}} \ar[r]^{(j_1^*,j_2^*)\quad\quad} &\Gamma(U,\mathcal{F}^{\bullet})\oplus \Gamma(V,\mathcal{F}^{\bullet})\ar[d]^{(PD_{U},PD_{V})}\ar[r]^{\quad\quad j_3^*-j_4^*}&\Gamma(U\cap V,\mathcal{F}^{\bullet})\ar[d]^{PD_{U\cap V}}\ar[r] &0\\
   0\ar[r] &\Gamma_c(U\cup V,\mathcal{G}^{\bullet})^\vee\ar[d]^{\textrm{id}} \quad\ar[r]^{(j_{1*}^\vee,j_{2*}^\vee)\quad\quad} &\quad\Gamma_c(U,\mathcal{G}^{\bullet})^\vee\oplus \Gamma_c(V,\mathcal{G}^{\bullet})^\vee\ar[d]_{\cong}^{I^\bullet}\quad\ar[r]^{\qquad\quad j_{3*}^\vee-j_{4*}^\vee}&\quad\Gamma_c(U\cap V,\mathcal{G}^{\bullet})^\vee\ar[d]^{\textrm{id}}\ar[r] &0\\
 0\ar[r] & \Gamma_c(U\cup V,\mathcal{G}^{\bullet})^\vee   \quad\ar[r]^{(j_{1*}-j_{2*})^\vee\quad\quad} & \quad\left(\Gamma_c(U,\mathcal{G}^{\bullet})\oplus \Gamma_c(V,\mathcal{G}^{\bullet})\right)^\vee\quad\ar[r]^{\quad\quad\quad (j_{3*},j_{4*})^\vee}&\quad\Gamma_c(U\cap V,\mathcal{G}^{\bullet})^\vee\ar[r] &0, }}
\end{displaymath}
where the maps $j_i$, $j_i^*$ and $j_{i*}$ are self-explanatory, for $i=1,2,3,4$. The first and third rows are exact sequences of complexes (seeing \cite{BT}, Prop. 2.3, 2.7). Since $I^\bullet$ is isomorphic, the second row is also an exact sequence. Therefore, we get a commutative diagram of long exact sequences
\begin{displaymath}
\tiny{\xymatrix{
  \cdots H^{p}(U\cup V,\mathcal{V})\ar[d]^{PD_{U\cup V}} \ar[r]^{(j_1^*,j_2^*)\quad\quad} &H^{p}(U,\mathcal{V})\oplus H^{p}(V,\mathcal{V})\ar[d]^{(PD_{U},PD_{V})}\ar[r]^{\quad\quad j_3^*-j_4^*}&H^{p}(U\cap V,\mathcal{V})\ar[d]^{PD_{U\cap V}}\ar[r]^{\quad\delta^p_{\mathcal{V}}} &\cdots\\
 \cdots H_c^{n-p}(U\cup V,\mathcal{V}^\vee)^\vee     \quad\ar[r]^{(j_{1*}-j_{2*})^\vee\qquad\quad} & \quad H_c^{n-p}(U,\mathcal{V}^\vee)^\vee\oplus H_c^{n-p}(V,,\mathcal{V}^\vee)^\vee\quad\ar[r]^{\qquad\qquad (j_{3*},j_{4*})^\vee}&\quad H_c^{n-p}(U\cap V,,\mathcal{V}^\vee)^\vee\ar[r]^{\qquad\qquad\quad(\delta^{n-p-1}_{\mathcal{V}^\vee})^\vee} &\cdots,}}
\end{displaymath}
which implies that $\mathcal{PD}:\mathcal{H}^*_X(\mathcal{V})\rightarrow (\mathcal{H}_{X,c}^{n-*}(\mathcal{V}^{\vee}))^{\vee}= \mathcal{H}om_{\underline{\mathbb{R}}_X}(\mathcal{H}_{X,c}^{n-*}(\mathcal{V}^{\vee}),\underline{\mathbb{R}}_X)$ is a morphsim of M-V systems of cdp presheaves by Section 2.2 $(1)$ $(iii)$. Choose a $\mathcal{V}$-constant basis $\mathfrak{U}$  of $X$. For any  $U_1$, $\ldots$, $U_l\in \mathfrak{U}$, $\mathcal{V}|_{\bigcap_{i=1}^l U_i}$ is constant. The classical Poincar\'{e} duality theorem say that $\mathcal{PD}(\bigcap_{i=1}^l U_i)=PD_{\bigcap_{i=1}^l U_i}$ is  isomorphic. By Theorem \ref{1.1}, $\mathcal{PD}$ is an isomorphism.  We get the \emph{Poincar\'{e} dulaity theorem}  for local systems as follows.
\begin{thm}\label{poincare-duality}
Let $X$ be a connected oriented smooth manifold with dimension $n$ and $\mathcal{V}$ a local system of $\mathbb{R}$-modules of finite rank on $X$. Then, for every $p$,
\begin{displaymath}
PD_X:H^p(X,\mathcal{V})\rightarrow (H_c^{n-p}(X,\mathcal{V}^{\vee}))^{\vee}
\end{displaymath}
is an isomorphism, where $PD_X([\alpha])([\beta])=\int_X \emph{tr}(\alpha\wedge\beta)$, for $[\alpha]\in H^p(X,\mathcal{V})$ and $[\beta]\in H_c^{n-p}(X,\mathcal{V}^{\vee})$.
\end{thm}
\begin{rem}
The Poincar\'{e} duality theorem is a special case of the Verdier duality theorem, refer to \cite{Dim}, Cor. 3.3.12.
\end{rem}

\subsection{K\"{u}nneth formula for compact vertical cohomology}
\begin{lem}\label{Lem-Kun}
Let $\mathcal{V}$ be a local system of $\mathbb{R}$-modules of finite rank on a smooth manifold $Y$.  Assume that $\emph{pr}_1$ and $\emph{pr}_2$ are projections from $\mathbb{R}^n\times Y$ onto $\mathbb{R}^n$ and $Y$  respectively. If $\mathbb{R}^n\times Y$ is viewed as a smooth fiber bundle over $\mathbb{R}^n$, then
\begin{displaymath}
\emph{pr}_2^*:H_c^*(Y,\mathcal{V})\tilde{\rightarrow} H_{cv}^*(\mathbb{R}^n\times Y,\emph{pr}_2^{-1}\mathcal{V})
\end{displaymath}
 is an isomorphism.
\end{lem}
\begin{proof}
By Proposition \ref{example4} $(2)$, $\textrm{pr}_2^*:\mathcal{H}_{Y,c}^*(\mathcal{V})\rightarrow^\prime\mathcal{H}_{\mathbb{R}^n\times Y,cv}^*(\textrm{pr}_2^{-1}\mathcal{V})$ is a M-V morphism. Let $\mathfrak{U}$ be a $\mathcal{V}$-constant basis of $Y$. For any  $U_1$, $\ldots$, $U_l\in \mathfrak{U}$, $\mathcal{V}|_{\bigcap_{i=1}^l U_i}$ is constant. By \cite{M1}, Thm. 3.7, $\textrm{pr}_2^*$ is  isomorphic on $\bigcap_{i=1}^l U_i$. By Theorem \ref{1.1}, we get the lemma.
\end{proof}
\begin{prop}\label{Kun}
Let $\mathcal{U}$  and $\mathcal{V}$ be local systems of $\mathbb{R}$-modules of finite ranks on smooth manifolds $X$ and $Y$ respectively. Assume that $\emph{pr}_1$ and $\emph{pr}_2$ are projections from $X\times Y$ onto $X$ and $Y$ respectively. If $H_{c}^*(Y,\mathcal{V})=\bigoplus_{p\geq0}H_{c}^p(Y,\mathcal{V})$ has finite dimension, then $\emph{pr}_1^*(\bullet)\cup \emph{pr}_2^*(\bullet)$  gives an isomorphism
\begin{displaymath}
\bigoplus_{p+q=l}H^p(X,\mathcal{U})\otimes_{\mathbb{R}}H_{c}^q(Y,\mathcal{V})\rightarrow H_{cv}^l(X\times Y,\mathcal{U}\boxtimes\mathcal{V})
\end{displaymath}
for any $l$, where $X\times Y$ is viewed as a smooth fiber bundle over $X$ and $\mathcal{U}\boxtimes\mathcal{V}=\emph{pr}_1^{-1}\mathcal{U}\otimes\emph{pr}_2^{-1}\mathcal{V}$ is the \emph{external tensor product} of $\mathcal{U}$ and $\mathcal{V}$.
\end{prop}
\begin{proof}
The proof is similar to that of  \cite{M1}, Thm. 3.2. For readers' convenience, we give a complete proof. Assume that $\textrm{dim}X=n$ and the open covering $\mathfrak{U}=\{U_\alpha\}$ of $X$ is both good   and  $\mathcal{V}$-constant. Define double complexes
\begin{displaymath}
K^{p,q}=\bigoplus_{r+s=q}C^p(\mathfrak{U},\mathcal{U}\otimes\mathcal{A}_X^r)\otimes_{\mathbb{R}} H_{c}^s(Y,\mathcal{V}),
\end{displaymath}
\begin{displaymath}
d'=\sum_{r+s=q}\delta\otimes \textrm{id}_{H_{c}^s(Y,\mathcal{V})},\quad d''=\sum_{r+s=q}\textmd{d}\otimes \textmd{id}_{H_{c}^s(Y,\mathcal{V})},
\end{displaymath}
and
\begin{displaymath}
L^{p,q}=C_{cv}^p(\textrm{pr}_1^{-1}\mathfrak{U},\mathcal{U}\boxtimes\mathcal{V}\otimes\mathcal{A}_{X\times Y}^q),
\end{displaymath}
\begin{displaymath}
d'=\delta,\quad d''=\textmd{d},
\end{displaymath}
where $(C_{\Phi}^\bullet(\mathfrak{U}, \mathcal{F}), \delta)$ denote the \u{C}ech complex of the sheaf $\mathcal{F}$ with supports in $\Phi$ associated to the open covering $\mathfrak{U}$.

Let $\{\gamma_i\}$ be a collection of forms in $\Gamma_c(\mathcal{V}\otimes\mathcal{A}^*_Y)$ with pure degrees, such that $\{[\gamma_i]\}$ is a basis of $H_{c}^*(Y,\mathcal{V})$. Linear extensions of the map
\begin{displaymath}
\{\eta_{\alpha_0,\ldots,\alpha_p}\}\otimes[\gamma_i]\mapsto \{pr_1^*\eta_{\alpha_0,\ldots,\alpha_p}\wedge pr_2^*\gamma_i\}.
\end{displaymath}
give a morphism $f:K^{*,*}\rightarrow L^{*,*}$ of double complexes. Set $\textrm{rank}\mathcal{V}=k$.  $H^r(U_{\alpha_0,...,\alpha_p},\mathcal{U})\cong H^r(U_{\alpha_0,...,\alpha_p})^{\oplus k}$ is $\mathbb{R}^k$ for $r=0$ and zero otherwise, since  $\mathcal{U}$ is constant on $U_{\alpha_0,...,\alpha_p}\cong\mathbb{R}^n$. So
\begin{displaymath}
\begin{aligned}
H_{d''}^q(K^{p,*})\cong &\left(\prod_{\alpha_0,\ldots,\alpha_p}H^0(U_{\alpha_0,...,\alpha_p},\mathcal{U})\right)\otimes_{\mathbb{R}}H_{c}^q(Y,\mathcal{V})\\
\cong &\prod_{\alpha_0,\ldots,\alpha_p}H_{c}^q(Y,\mathcal{V})^{\oplus k},
\end{aligned}
\end{displaymath}
where the finiteness of the dimension of $H_{c}^s(Y,\mathcal{V})$ implies the second isomorphism, and
\begin{displaymath}
\begin{aligned}
H_{d''}^q(L^{p,*})= &\prod_{\alpha_0,\ldots,\alpha_p}H_{cv}^q(U_{\alpha_0,\ldots,\alpha_p}\times Y,\mathcal{U}\boxtimes\mathcal{V})\\
\cong &\prod_{\alpha_0,\ldots,\alpha_p}H_{cv}^q(U_{\alpha_0,\ldots,\alpha_p}\times Y,\textrm{pr}_2^{-1}\mathcal{V})^{\oplus k}.
\end{aligned}
\end{displaymath}
The morphism $H_{d''}^q(K^{p,*})\rightarrow H_{d''}^q(L^{p,*})$ induced by $f$ is just
\begin{displaymath}
\prod_{\alpha_0,\ldots,\alpha_p} (\textrm{pr}_2^*)^{\oplus k}:\prod_{\alpha_0,\ldots,\alpha_p}H_{c}^q(Y,\mathcal{V})^{\oplus k}\rightarrow \prod_{\alpha_0,\ldots,\alpha_p}H_{cv}^q(U_{\alpha_0,\ldots,\alpha_p}\times Y,\textrm{pr}_2^{-1}\mathcal{V})^{\oplus k},
\end{displaymath}
which is isomorphic by Lemma \ref{Lem-Kun}. Therefore, $f$ induces an isomorphism $H^p(f):H^p(K^*)\rightarrow H^p(L^*)$, where $K^*$ and $L^*$ are the simple complexes associated to $K^{*,*}$ and $L^{*,*}$, respectively.

Consider the spectral sequence $E_2^{p,q}=H_{d''}^qH_{d'}^p(K^{*,*})\Rightarrow H^{p+q}(K^*)$. By Leray Theorem (\cite{Dem}, p. 209, (5.17)),
\begin{displaymath}
E_2^{p,q}=\left\{
 \begin{array}{ll}
 \bigoplus_{r+s=q}H^r(X,\mathcal{U})\otimes_{\mathbb{R}}H_{c}^s(Y,\mathcal{V}),&~p=0\\
 &\\
 0,&~\textrm{otherwise},
 \end{array}
 \right.
\end{displaymath}
which implies $H^l(K^*)=\bigoplus_{p+q=l}H^p(X,\mathcal{U})\otimes_{\mathbb{R}}H_{c}^q(Y,\mathcal{V})$. Through similar computations, $H^l(L^*)=H_{cv}^l(X\times Y,\mathcal{U}\boxtimes\mathcal{V})$. We complete the proof.
\end{proof}

\begin{rem}
K\"{u}nneth formulas for cohomology and cohomology with compact support can be proved as  \cite{M1}, Theorem 3.2.
\end{rem}

\subsection{Leray-Hirsch theorem}
For graded vector spaces $V^*$ and $W^*$ over $\mathbb{R}$, the graded vector space $V^*\otimes_{\mathbb{R}}W^*$ is defined as
\begin{displaymath}
(V^*\otimes_{\mathbb{R}}W^*)^k=\bigoplus_{p+q=k}V^p\otimes_{\mathbb{R}}W^q
\end{displaymath}
for any $k$. For  bigraded vector spaces $V^{*,*}$ and $W^{*,*}$ over  $\mathbb{R}$, the bigraded vector space $V^{*,*}\otimes_{\mathbb{R}}W^{*,*}$ is defined as
\begin{displaymath}
(V^{*,*}\otimes_{\mathbb{R}}W^{*,*})^{k,l}=\bigoplus\limits_{\substack{p+q=k\\r+s=l}}V^{p,r}\otimes_{\mathbb{R}}W^{q,s}
\end{displaymath}
for any $k$, $l$.

\begin{thm}\label{L-H}
$(1)$ Let $\pi:E\rightarrow X$ be a smooth fiber bundle over a connected smooth manifold $X$ and $\mathcal{V}$ a local system of $\mathbb{R}$-modules of finite rank on $X$.

$(i)$ Assume that there exist classes $e_1,\dots,e_r$ of pure degrees in $H^*(E,\mathbb{R})$, such that, for every $x\in X$,  their restrictions $e_1|_{E_x},\dots,e_r|_{E_x}$ freely linearly generate $H^*(E_x,\mathbb{R})$. Then, $\pi^*(\bullet)\cup\bullet$ gives an isomorphism of graded vector spaces
\begin{displaymath}
H^*(X,\mathcal{V})\otimes_{\mathbb{R}} \emph{span}_{\mathbb{R}}\{e_1, ..., e_r\} \tilde{\rightarrow} H^*(E,\pi^{-1}\mathcal{V}),
\end{displaymath}
where $\emph{span}_{\mathbb{R}}\{e_1, ..., e_r\}$ is viewed as a graded subspace of $H^*(E,\mathbb{R})$.

$(ii)$ Assume that there exist classes $e_1,\dots,e_r$ of pure degrees in $ H_{cv}^*(E,\mathbb{R})$, such that, for every $x\in X$,  their restrictions $e_1|_{E_x},\dots,e_r|_{E_x}$ freely linearly generate $H_{c}^*(E_x,\mathbb{R})$. Then, $\pi^*(\bullet)\cup\bullet$ gives isomorphisms of graded vector spaces
\begin{displaymath}
H_c^*(X,\mathcal{V})\otimes_{\mathbb{R}} \emph{span}_{\mathbb{R}}\{e_1, ..., e_r\} \tilde{\rightarrow} H_c^*(E,\pi^{-1}\mathcal{V})
\end{displaymath}
and
\begin{displaymath}
H^*(X,\mathcal{V})\otimes_{\mathbb{R}} \emph{span}_{\mathbb{R}}\{e_1, ..., e_r\} \tilde{\rightarrow} H_{cv}^*(E,\pi^{-1}\mathcal{V}),
\end{displaymath}
where $\emph{span}_{\mathbb{R}}\{e_1, ..., e_r\}$ is viewed as a graded subspace of $H_{cv}^*(E,\mathbb{R})$.

$(iii)$ Assume that there exist classes $e_1,\dots,e_r$ of pure degrees in $ H_{c}^*(E,\mathbb{R})$, such that, for every $x\in X$,  their restrictions $e_1|_{E_x},\dots,e_r|_{E_x}$ freely linearly generate $H_{c}^*(E_x,\mathbb{R})$. Then, $\pi^*(\bullet)\cup\bullet$ gives an isomorphism of graded vector spaces
\begin{displaymath}
H_c^*(X,\mathcal{V})\otimes_{\mathbb{R}} \emph{span}_{\mathbb{R}}\{e_1, ..., e_r\} \tilde{\rightarrow} H_c^*(E,\pi^{-1}\mathcal{V}),
\end{displaymath}
where $\emph{span}_{\mathbb{R}}\{e_1, ..., e_r\}$ is viewed as a graded subspace of $H_c^*(E,\mathbb{R})$.

$(2)$ Let $\pi:E\rightarrow X$ be a holomorphic fiber bundle over a connected complex manifold $X$ and let $\mathcal{E}$ be a locally free sheaf of $\mathcal{O}_X$-modules of finite rank on $X$. Assume that there exist classes $e_1,\dots,e_r$ of pure degrees in $ H^{*,*}(E)$, such that, for every $x\in X$,  their restrictions $e_1|_{E_x},\dots,e_r|_{E_x}$ freely linearly generate $H^{*,*}(E_x)$. Then, $\pi^*(\bullet)\cup\bullet$ gives an isomorphism of bigraded vector spaces
\begin{displaymath}
H^{*,*}(X,\mathcal{E})\otimes_{\mathbb{C}} \emph{span}_{\mathbb{C}}\{e_1, ..., e_r\} \tilde{\rightarrow} H^{*,*}(E,\pi^*\mathcal{E}),
\end{displaymath}
where $\emph{span}_{\mathbb{R}}\{e_1, ..., e_r\}$ is viewed as a bigraded subspace of $H^{*,*}(E,\mathbb{R})$.
\end{thm}
\begin{proof}
$(1)$ Set $k_i=\textrm{deg}e_i$, for $0\leq i\leq r$.

$(i)$ By Proposition \ref{example2} $(2)$, $(3)$ $(i)$,
\begin{equation}\label{L-H-1}
\pi^*:\mathcal{H}^{*}_X(\mathcal{V})[-k_i] \rightarrow \pi_*\mathcal{H}^{*}_E(\pi^{-1}\mathcal{V})[-k_i]
\end{equation}
and
\begin{equation}\label{L-H-2}
e_i\cup:\mathcal{H}^{*}_E(\pi^{-1}\mathcal{V})[-k_i] \rightarrow \mathcal{H}^*_E(\pi^{-1}\mathcal{V})
\end{equation}
are morphisms of M-V systems of cdp presheaves. Push out (\ref{L-H-2}), by Proposition \ref{elem} $(3)$,
\begin{equation}\label{L-H-3}
e_i\cup:\pi_*\mathcal{H}^{*}_E(\pi^{-1}\mathcal{V})[-k_i] \rightarrow \pi_*\mathcal{H}^*_E(\pi^{-1}\mathcal{V})
\end{equation}
is also a morphism of M-V systems of cdp presheaves, so is the composition of (\ref{L-H-1}) and (\ref{L-H-3})
\begin{displaymath}
\pi^*(\bullet)\cup e_i:\mathcal{H}^{*}_X(\mathcal{V})[-k_i] \rightarrow \pi_*\mathcal{H}^*_E(\pi^{-1}\mathcal{V})
\end{displaymath}
by  Section 2.2 $(5')$. Hence the sum
\begin{displaymath}
F^*=\sum_{i=1}^r\pi^*(\bullet)\cup e_i:\bigoplus_{i=1}^r\mathcal{H}^{*}_X(\mathcal{V})[-k_i] \rightarrow \pi_*\mathcal{H}^*_E(\pi^{-1}\mathcal{V})
\end{displaymath}
is a M-V morphism by Section 2.2 $(4')$. Let $\mathfrak{U}$ be a $\mathcal{V}$-constant basis of $X$. For $U_1$, ..., $U_l\in\mathfrak{U}$,  $\mathcal{V}|_{U_1\cap...\cap U_l}$ is constant. By \cite{M1}, Thm. 3.10 (1), $F^*$ is an isomorphism on $U_1\cap...\cap U_l$. By Theorem \ref{1.1}, we get $(i)$ immediately.

$(ii)$ As the proof of $(i)$, by Proposition \ref{example3} $(3)$, Proposition \ref{elem} $(3)$ and Section 2.2 $(4')$, $(5')$,
\begin{displaymath}
\sum_{i=1}^r\pi^*(\bullet)\cup e_i:\bigoplus_{i=1}^r\mathcal{H}^{*}_{X,c}(\mathcal{V})[-k_i]\rightarrow \pi_*\mathcal{H}^*_{E,c}(\pi^{-1}\mathcal{V})
\end{displaymath}
is a morphism of M-V systems of cds precosheaves. By Proposition \ref{example2} $(3)$ $(i)$, Proposition \ref{example3} $(2)$, Proposition \ref{elem} $(5)$ and Section 2.2 $(4)$, $(5')$,
\begin{displaymath}
\sum_{i=1}^r\pi^*(\bullet)\cup e_i:\bigoplus_{i=1}^r\mathcal{H}^{*}_X(\mathcal{V})[-k_i] \rightarrow \pi_*\mathcal{H}^*_{E,cv}(\pi^{-1}\mathcal{V})
\end{displaymath}
is a morphism of M-V systems of cdp presheaves. The rest of the proof is almost the same as that of $(i)$, except that we use \cite{M1}, Thm. 3.10 (2) instead of \cite{M1}, Thm. 3.10 (1).

$(iii)$ Let $\hat{e}_1$, $\ldots$, $\hat{e}_r$ be the images of $e_1$, $\ldots$, $e_r$ under the natural map $H_{c}^*(E)\rightarrow H_{cv}^*(E)$. By the hypothesis,  the natural map gives an isomorphism $\textrm{span}_{\mathbb{R}}\{e_1, ..., e_r\}\tilde{\rightarrow}\textrm{span}_{\mathbb{R}}\{\hat{e}_1, ..., \hat{e}_r\}$, which immediately implies $(iii)$ by $(ii)$.

$(2)$ For $0\leq i\leq r$, assume that the bidegree of $e_i$ is $(k_i,l_i)$. Fixed $p\in\mathbb{Z}$. By Proposition \ref{example1} $(2)$, $(3)$, Proposition \ref{elem} $(3)$ and Section 2.2 $(4')$, $(5')$,
\begin{displaymath}
F^*=\sum_{i=1}^r\pi^*(\bullet)\cup e_i:\bigoplus_{i=1}^r\mathcal{H}^{p-k_i,*}_X(\mathcal{E})[-l_i]\rightarrow \pi_*\mathcal{H}^{p,*}_E(\pi^*\mathcal{E}).
\end{displaymath}
is a morphism of M-V systems of cdp presheaves. As the proof of $(1)$ $(i)$, by \cite{M2}, Thm. 1.2, $F^*$ satisfies the hypothesis $(*)$ in Theorem \ref{1.1}, hence we proved $(2)$.
\end{proof}

Immediately, we get the projective bundle formulas.
\begin{cor}\label{proj-bun}
Let $\pi:\mathbb{P}(E)\rightarrow X$ be the projective bundle associated to a complex vector bundle $E$ of rank $r$ on a complex manifold $X$. Set $t=\frac{i}{2\pi}\Theta(\mathcal{O}_{\mathbb{P}(E)}(-1))\in \mathcal{A}^2({\mathbb{P}(E)})$, where $\mathcal{O}_{\mathbb{P}(E)}(-1)$ is the universal line bundle  on ${\mathbb{P}(E)}$ and  $\Theta(\mathcal{O}_{\mathbb{P}(E)}(-1))$ is a curvature of a hermitian metric on $\mathcal{O}_{\mathbb{P}(E)}(-1)$.

$(1)$ For any local system  $\mathcal{V}$ of $\mathbb{R}$-modules of finite rank, $\pi^*(\bullet)\cup\bullet$ gives isomorphisms of graded vector spaces
\begin{displaymath}
H^*(X,\mathcal{V})\otimes_{\mathbb{R}}\emph{span}_{\mathbb{R}}\{1,...,h^{r-1}\}\tilde{\rightarrow}H^*(\mathbb{P}(E),\pi^{-1}\mathcal{V})
\end{displaymath}
and
\begin{displaymath}
H_{c}^*(X,\mathcal{V})\otimes_{\mathbb{R}}\emph{span}_{\mathbb{R}}\{1,...,h^{r-1}\}\tilde{\rightarrow}H_{c}^*(\mathbb{P}(E),\pi^{-1}\mathcal{V}),
\end{displaymath}
where $h=[t]\in H^2({\mathbb{P}(E)},\mathbb{R})$ and $\emph{span}_{\mathbb{R}}\{1,...,h^{r-1}\}$ is viewed as a graded subspace of $H^*(\mathbb{P}(E),\mathbb{R})$.

$(2)$ Assume that $E$ is holomorphic and $\Theta(\mathcal{O}_{\mathbb{P}(E)}(-1))\in \mathcal{A}^{1,1}(\mathbb{P}(E))$ is the Chern curvature of a hermitian metric on $\mathcal{O}_{\mathbb{P}(E)}(-1)$. For any locally free sheaf  $\mathcal{E}$  of $\mathcal{O}_X$-modules of finite rank, $\pi^*(\bullet)\cup\bullet$ gives an isomorphism of bigraded vector spaces
\begin{displaymath}
H^{*,*}(X,\mathcal{E})\otimes_{\mathbb{C}}\emph{span}_{\mathbb{C}}\{1,...,h^{r-1}\}\tilde{\rightarrow}H^{*,*}(\mathbb{P}(E),\pi^*\mathcal{E}),
\end{displaymath}
where $h=[t]\in H^{1,1}({\mathbb{P}(E)})$ and $\emph{span}_{\mathbb{C}}\{1,...,h^{r-1}\}$ is viewed as a bigraded subspace of $H^{*,*}(\mathbb{P}(E))$.
\end{cor}
\begin{rem}
Following the steps of \cite{CFGU},  Rao, S., Yang, S. and Yang, X.-D. gave another version of Hirsch theorem  (\cite{RYY2}, Lemma 3.2), where they assumed that $ H^{*,*}(F)$ is a free bigraded algebra. Using this, they  proved Corollary \ref{proj-bun} (2) (\cite{RYY2}, Lemma 3.3).
\end{rem}

Let $\pi:E\rightarrow X$  be an oriented smooth vector bundle of rank $r$ on a $($not necessarily orientable$)$ smooth manifold $X$. By \cite{BT}, Thm. 6.17 and Remark 6.17.1,
$\pi_*:H_{cv}^*(E,\mathbb{R})\rightarrow H^{*-r}(X,\mathbb{R})$ is an isomorphism. There exists  a closed form $\Phi\in\mathcal{A}_{cv}^r(E)$ satisfying $\pi_*[\Phi]_{cv}=1$ in $H^0(X,\mathbb{R})$, i.e.,  $[\Phi]_{cv}\in H_{cv}^r(E,\mathbb{R})$ is the \emph{Thom class} of $E$. Clearly, $\pi_*\Phi=1$ in $\mathcal{A}^0(X)$. We get the \emph{Thom isomorphism theorem} for local systems as follows.

\begin{cor}\label{is}
Let $\pi:E\rightarrow X$  be an oriented smooth vector bundle of rank $r$ on a $($not necessarily orientable$)$ smooth manifold $X$. Assume that $\mathcal{V}$ is a local system of $\mathbb{R}$-modules of finite rank on $X$. Then $\Phi\wedge \pi^*(\bullet)$ gives isomorphisms
\begin{displaymath}
H_{c}^{*-r}(X,\mathcal{V})\tilde{\rightarrow} H_{c}^*(E,\pi^{-1}\mathcal{V})
\end{displaymath}
and
\begin{displaymath}
H^{*-r}(X,\mathcal{V})\tilde{\rightarrow} H_{cv}^*(E,\pi^{-1}\mathcal{V}),
\end{displaymath}
which both have the inverse isomorphism $\pi_*$. Moreover, if $X$ is oriented, they coincide with the pushout $i_*$.
\end{cor}
\begin{proof}
By \cite{BT}, Prop. 6.18, the restriction $[\Phi]|_{E_x}$ is a generator of $H_{c}^*(E_x,\mathbb{R})=\mathbb{R}$. By Theorem \ref{L-H} $(1)$ $(ii)$, $\Phi\wedge \pi^*(\bullet)$ gives the two isomorphisms. For any $\mathcal{V}$-valued form $\omega$ on $X$, $\pi_*(\Phi\wedge\pi^*\omega)=\omega$ by the projection formula (\ref{pro-formula5}). So $\pi_*$ is their inverse isomorphisms.

If $X$ is oriented, the pushout $i_*$ is defined well in both cases (seeing Section 3.1.3 and 3.2) and $\pi_*i_*=\textrm{id}$. So $i_*=\pi_*^{-1}=\Phi\wedge\pi^*$.
\end{proof}

\subsection{Formulas of proper modifications}
Recall that a proper holomorphic map $\pi:X\rightarrow Y$ between connected complex manifolds is called a \emph{proper modification}, if there is a nowhere dense analytic subset $F\subset Y$, such that $\pi^{-1}(F)\subset X$ is nowhere dense and $\pi:X-f^{-1}(F)\rightarrow Y-F$ is biholomorphic.
\begin{prop}\label{prop}
Let $\pi:X\rightarrow Y$ be a proper modification of $n$-dimensional connected complex manifolds. Suppose that $\mathcal{V}$ is a local system of $\mathbb{R}$-modules of finite rank on $Y$ and $\mathcal{E}$ is a locally free sheaf of $\mathcal{O}_Y$-modules of finite rank on $Y$. Then $\pi^*$ gives isomorphisms of graded vector spaces

$(1)$ $H_{c}^{2n-1}(Y,\mathcal{V})\tilde{\rightarrow}H_{c}^{2n-1}(X,\pi^{-1}\mathcal{V})$,

$(2)$ $H^1(Y,\mathcal{V})\tilde{\rightarrow}H^1(X,\pi^{-1}\mathcal{V})$,

$(3)$ $H^{0,*}(Y,\mathcal{E})\tilde{\rightarrow} H^{0,*}(X,\pi^*\mathcal{E})$,

$(4)$ $H_c^{0,*}(Y,\mathcal{E})\tilde{\rightarrow} H_c^{0,*}(X,\pi^*\mathcal{E})$,

$(5)$ $H^{*,0}(Y,\mathcal{E})\tilde{\rightarrow} H^{*,0}(X,\pi^*\mathcal{E})$,

$(6)$ $H_c^{*,0}(Y,\mathcal{E})\tilde{\rightarrow} H_c^{*,0}(X,\pi^*\mathcal{E})$.
\end{prop}
\begin{proof}
We will only prove $(3)$. By Proposition \ref{example1} $(3)$, $\pi^*:\mathcal{H}_{Y,c}^{0,*}(\mathcal{E})\rightarrow \pi_*\mathcal{H}_{X,c}^{0,*}(\pi^*\mathcal{E})$ is a  morphism of M-V systems of cds precosheaves. Set $\mathfrak{U}$ a $\mathcal{E}$-free basis of $Y$. For $U_1$, ..., $U_l\in \mathfrak{U}$, $\mathcal{E}|_{\bigcap_{i=1}^l U_i}$ is free. By  \cite{M3}, Prop. 4.3, $\pi^*$ is isomorphic on $\bigcap_{i=1}^l U_i$, i.e., $\pi^*$ satisfies the hypothesis $(*)$ in Theorem \ref{1.1}. So $\pi^*:H_c^{0,*}(Y,\mathcal{E})\rightarrow H_c^{0,*}(X,\pi^*\mathcal{E})$ is an isomorphism.

Other cases can be proved similarly, where we use \cite{M1}, Thm. 4.7. or \cite{M3}, Prop. 4.4 to check the hypothesis $(*)$ in Theorem \ref{1.1}.
\end{proof}
\begin{rem}
If $\pi$ is a blow-up map, Proposition \ref{prop} can be obtained immediately by Theorem \ref{blow-up} in the following section.
\end{rem}

\section{Blow-up formulas}
Blow-ups are fundamental transforms on complex manifolds and algebraic varieties.  They play an important role in complex and algebraic geometries. It is significant to calculate various cohomologies of blow-ups. There are several results of this aspect, seeing \cite{Gri,Hu,Ma,M1,M2,RYY,RYY2,V,YY,YZ}. In follows, we study blow-up formulas for cohomology with values in local systems and Dolbeault cohomology with values in locally free sheaves.

Let $\pi:\widetilde{X}\rightarrow X$ be the blow-up of a connected complex manifold $X$ along a connected complex submaifold $Y$. We know $\pi|_E:E=\pi^{-1}(Y)\rightarrow Y$ is the projective bunde $\mathbb{P}(N_{Y/X})$ associated to the normal bundle $N_{Y/X}$ over $Y$. Assume that $i_Y:Y\rightarrow X$ and $i_E:E\rightarrow \widetilde{X}$ are inclusions and $r=\textrm{codim}_{\mathbb{C}}Y$. Set $t=\frac{i}{2\pi}\Theta(\mathcal{O}_{E}(-1))\in \mathcal{A}^{1,1}(E)$, where $\mathcal{O}_{E}(-1)$ is the universal line bundle  on $E={\mathbb{P}(N_{Y/X})}$ and  $\Theta(\mathcal{O}_{E}(-1))$ is the Chern curvature of a hermitian metric on $\mathcal{O}_{E}(-1)$. Clearly, $\textrm{d}t=0$ and $\bar{\partial}t=0$.
\subsection{Blow-up formulas}
\begin{thm}\label{blow-up}
Assume that $X$, $Y$, $\pi$, $\widetilde{X}$, $E$, $t$, $i_Y$, $i_E$ are defined as above.  Then,
\begin{displaymath}\label{b-u-m}
\pi^*+\sum_{i=0}^{r-2}(i_E)_*\circ (h^i\cup)\circ (\pi|_E)^*
\end{displaymath}
gives isomorphisms
\begin{equation}\label{1}
H^k(X,\mathcal{V})\oplus \bigoplus_{i=0}^{r-2}H^{k-2-2i}(Y,i_Y^{-1}\mathcal{V})\tilde{\rightarrow} H^k(\widetilde{X}, \pi^{-1}\mathcal{V}),
\end{equation}
\begin{equation}\label{2}
H_{c}^k(X,\mathcal{V})\oplus \bigoplus_{i=0}^{r-2}H_{c}^{k-2-2i}(Y,i_Y^{-1}\mathcal{V})\tilde{\rightarrow} H_{c}^k(\widetilde{X},\pi^{-1}\mathcal{V}),
\end{equation}
for any $k$, where $h=[t]\in H^2(E,\mathbb{R})$ or $H^2(E,\mathbb{C})$  and $\mathcal{V}$ is a local system of $\mathbb{R}$ or $\mathbb{C}$-modules of finite rank on $X$, and an isomorphism
\begin{equation}\label{3}
H^{p,q}(X,\mathcal{E})\oplus \bigoplus_{i=0}^{r-2}H^{p-1-i,q-1-i}(Y,i_Y^*\mathcal{E})\tilde{\rightarrow} H^{p,q}(\widetilde{X},\pi^*\mathcal{E}),
\end{equation}
for any $p$, $q$, where $h=[t]\in H^{1,1}(E)$ and $\mathcal{E}$ is a locally free sheaf of $\mathcal{O}_X$-modules of finite rank on $X$.
\end{thm}
\begin{proof}
By Proposition \ref{example1} $(3)$,
\begin{displaymath}
\pi^*:\mathcal{H}^{p,*}_X(\mathcal{E})\rightarrow \pi_*\mathcal{H}_{\widetilde{X}}^{p,*}(\pi^*\mathcal{E}),
\end{displaymath}
\begin{displaymath}
(\pi|_E)^*:\mathcal{H}_Y^{p-1-i,*}(i_Y^*\mathcal{E})[-1-i]\rightarrow (\pi|_E)_*\mathcal{H}_{E}^{p-1-i,*}((\pi|_E)^*i_Y^*\mathcal{E})[-1-i],
\end{displaymath}
\begin{displaymath}
(i_E)_*:i_{E*}\mathcal{H}_{E}^{p-1,*}(i_E^*\pi^*\mathcal{E})[-1]\rightarrow \mathcal{H}_{\widetilde{X}}^{p,*}(\pi^*\mathcal{E})
\end{displaymath}
and, by Proposition \ref{example1} $(2)$,
\begin{displaymath}
h^i\cup:\mathcal{H}_{E}^{p-1-i,*}((\pi|_E)^*i_Y^*\mathcal{E})[-1-i]\rightarrow \mathcal{H}_{E}^{p-1,*}((\pi|_E)^*i_Y^*\mathcal{E})[-1]
\end{displaymath}
are  morphisms of M-V systems of cdp presheaves, so are
\begin{equation}\label{a}
(\pi|_E)^*:i_{Y*}\mathcal{H}_Y^{p-1-i,*}(i_Y^*\mathcal{E})[-1-i]\rightarrow i_{Y*}(\pi|_E)_*\mathcal{H}_{E}^{p-1-i,*}((\pi|_E)^*i_Y^*\mathcal{E})[-1-i],
\end{equation}
\begin{equation}\label{b}
\begin{aligned}
h^i\cup:i_{Y*}(\pi|_E)_*\mathcal{H}_{E}^{p-1-i,*}((\pi|_E)^*i_Y^*\mathcal{E})[-1-i]\rightarrow & i_{Y*}(\pi|_E)_*\mathcal{H}_{E}^{p-1,*}((\pi|_E)^*i_Y^*\mathcal{E})[-1]\\
=&\pi_*i_{E*}\mathcal{H}_{E}^{p-1,*}(i_E^*\pi^*\mathcal{E})[-1],
\end{aligned}
\end{equation}
and
\begin{equation}\label{c}
(i_E)_*:\pi_*i_{E*}\mathcal{H}_{E}^{p-1,*}(i_E^*\pi^*\mathcal{E})[-1]\rightarrow \pi_*\mathcal{H}_{\widetilde{X}}^{p,*}(\pi^*\mathcal{E})
\end{equation}
by Proposition \ref{elem} $(1)$, $(3)$. Combining  (\ref{a})-(\ref{c}), we obtain a M-V morphism
\begin{displaymath}
(i_E)_*\circ (h^i\cup)\circ (\pi|_E)^*:i_{Y*}\mathcal{H}_Y^{p-1-i,*}(i_Y^*\mathcal{E})[-1-i]\rightarrow \pi_*\mathcal{H}_{\widetilde{X}}^{p,*}(\pi^*\mathcal{E})
\end{displaymath}
by Section 2.2 $(5')$. By Section 2.2 $(2)$, $(4')$, the sum
\begin{displaymath}
F^*=\pi^*+\sum_{i=0}^{r-2}(i_E)_*\circ (h^i\cup)\circ (\pi|_E)^*
\end{displaymath}
gives a  morphism
\begin{displaymath}
\mathcal{H}^{p,*}_X(\mathcal{E})\oplus \bigoplus_{i=0}^{r-2}i_{Y*}\mathcal{H}_Y^{p-1-i,*}(i_Y^*\mathcal{E})[-1-i]\rightarrow \pi_*\mathcal{H}_{\widetilde{X}}^{p,*}(\pi^*\mathcal{E})
\end{displaymath}
of M-V systems of cdp presheaves.

Let $\mathfrak{U}$ be a $\mathcal{E}$-free basis of $X$. For any $U_1$, ..., $U_l\in \mathfrak{U}$, $\mathcal{E}|_{\bigcap_{i=1}^l U_i}$ is free. By \cite{M2}, Thm. 1.2, $F^*$ is isomorphic on $\bigcap_{i=1}^l U_i$. Hence, $F^*$ satisfies (*) in Theorem \ref{1.1} and then $F^*$ is an isomorphism. We proved (\ref{3}).

The proofs of (\ref{1}) and (\ref{2}) are almost the same as that of (\ref{3}), except that we use   Proposition \ref{example2} and \cite{M1}, Thm. 1.3 instead of Proposition \ref{example1} and \cite{M2}, Thm. 1.2.
\end{proof}

\subsection{Comparison with the formula  given by Rao, S., Yang, S. and Yang, X.-D.}
\subsubsection{\emph{\textbf{The formula  given by Rao, S. et al.}}}
On compact complex manifolds, Rao, S., Yang, S. and Yang, X.-D. (\cite{RYY2}) gave an explicit formula of blow-ups for bundle-valued Dolbeault cohomology. We recall their construction as follows:

Suppose that $X$ is compact. For any $\alpha$ in $H^{p,q}(\widetilde{X},\pi^*\mathcal{E})$, by \cite{RYY2}, Lemma 3.3, there exist unique $\alpha^{p-i,q-i}$ in  $H^{p-i,q-i}(Y,i_Y^*\mathcal{E})$, for  $i=0,...,r-1$ such that
\begin{equation}\label{rep}
i_E^*\alpha =\sum_{i=0}^{r-1}h^i\cup (\pi|_E)^*\alpha^{p-i,q-i}.
\end{equation}
Define
\begin{equation}\label{def}
\phi:H^{p,q}(\widetilde{X},\pi^*\mathcal{E})\rightarrow H^{p,q}(X,\mathcal{E})\oplus \bigoplus_{i=1}^{r-1}H^{p-i,q-i}(Y,i_Y^*\mathcal{E})
\end{equation}
\begin{displaymath}
\quad\quad\quad\quad\quad\alpha\mapsto(\pi_*\alpha,\alpha^{p-1,q-1},...,\alpha^{p-r+1,q-r+1}).
\end{displaymath}
Clearly, $\phi$ is a linear map of vector spaces. Rao, S. et al. proved that $\phi$ is injective and
\begin{equation}\label{iso}
H^{p,q}(\widetilde{X},\pi^*\mathcal{E})\cong H^{p,q}(X,\mathcal{E})\oplus \bigoplus_{i=1}^{r-1}H^{p-i,q-i}(Y,i_Y^*\mathcal{E}),
\end{equation}
which imply that  $\phi$ is an isomorphism.

From now on, assume that $X$ is a (not necessarily compact) connected complex manifold. We notice that, for the definition (\ref{def}) of $\phi$, the compactness of $X$  is not necessary, since we can use Corollary \ref{proj-bun} in the present paper instead of Lemma 3.3 in \cite{RYY2} for  general cases. With the same proof in \cite{RYY2}, the injectivity of $\phi$ and the isomorphism (\ref{iso}) also hold for the noncompact manifold $X$. However, we don't know whether $\phi$ is isomorphic in this case, since the dimensions of cohomologies  are possibly infinite. In follows, through comparison of $\phi$ and the formula given in Theorem \ref{blow-up}, we will prove that $\phi$ is still isomorphic on general complex manifolds.
\subsubsection{\emph{\textbf{Relative Dolbeault sheaves}}}
Recall two sheaves defined in \cite{RYY} and their properties. For more details, we refer to \cite{RYY2}, Sec. 4.2 for compact cases, or \cite{M2}, Sec. 4 for general cases.

Let $X$ be a connected complex manifold and $i:Y\rightarrow X$ the inclusion of a closed complex submanifold $Y$ into $X$. For any $p$, $q$, set
$\mathcal{F}_{X,Y}^{p,q}=\textrm{ker}(\mathcal{A}_X^{p,q}\rightarrow i_*\mathcal{A}_Y^{p,q})$. We have an exact sequence of sheaves
\begin{equation}\label{relative}
0\rightarrow\mathcal{F}_{X,Y}^{p,q}\rightarrow\mathcal{A}_X^{p,q}\rightarrow i_*\mathcal{A}_Y^{p,q}\rightarrow 0,
\end{equation}
for any $p$. Define
$\mathcal{F}_{X,Y}^p=\textrm{ker}(\bar{\partial}:\mathcal{F}_{X,Y}^{p,0}\rightarrow\mathcal{F}_{X,Y}^{p,1})$.
There is a resolution of soft sheaves of $\mathcal{F}_{X,Y}^p$
\begin{displaymath}
\xymatrix{
0\ar[r] &\mathcal{F}_{X,Y}^p\ar[r]^{i} &\mathcal{F}_{X,Y}^{p,0}\ar[r]^{\bar{\partial}} &\mathcal{F}_{X,Y}^{p,1}\ar[r]^{\bar{\partial}}&\cdots\ar[r]^{\bar{\partial}}&\mathcal{F}_{X,Y}^{p,n}\ar[r]&0.
}
\end{displaymath}
$\mathcal{F}_{X,Y}^p$ and $\mathcal{F}_{X,Y}^{p,q}$ are called the \emph{relative Dolbeault sheaves} of $X$ with respect to $Y$. For a locally free sheaf $\mathcal{E}$  of $\mathcal{O}_X$-modules of finite rank on $X$, we get an exact sequence of sheaves
\begin{displaymath}
\xymatrix{
 0\ar[r]&\mathcal{E}\otimes\mathcal{F}_{X,Y}^{p,q} \ar[r]& \mathcal{E}\otimes\mathcal{A}_{X}^{p,q} \ar[r]^{}& i_*(i^*\mathcal{E}\otimes\mathcal{A}_{Y}^{p,q}) \ar[r]& 0}
\end{displaymath}
by  (\ref{relative}) and the projection formula of sheaves. Since $\mathcal{E}\otimes\mathcal{F}_{X,Y}^{p,q}$ is $\Gamma$-acyclic,
\begin{equation}\label{relative-exact}
\small{\xymatrix{
 0\ar[r]&\Gamma(X,\mathcal{E}\otimes\mathcal{F}_{X,Y}^{p,q}) \ar[r]& \Gamma(X,\mathcal{E}\otimes\mathcal{A}_{X}^{p,q}) \ar[r]^{i^*}& \Gamma(Y,i^*\mathcal{E}\otimes\mathcal{A}_{Y}^{p,q}) \ar[r]& 0}}
\end{equation}
is exact.

Now, we go back to the cases of blow-ups. By \cite{M2}, Lemma 4.2,
\begin{displaymath}
R^q\pi_*(\pi^*\mathcal{E}\otimes\mathcal{F}_{\widetilde{X},E}^{p})=\mathcal{E}\otimes R^q\pi_*\mathcal{F}^p_{\widetilde{X},E}=\left\{
 \begin{array}{ll}
\mathcal{E}\otimes\mathcal{F}^p_{X,Y},&~q=0\\
 &\\
 0,&~q\geq1.
 \end{array}
 \right.
\end{displaymath}
By the Leray spectral sequence, $\pi^*$ induces an isomorhism
\begin{displaymath}
H^q(X,\mathcal{E}\otimes\mathcal{F}_{X,Y}^p)\cong H^q(\widetilde{X},\pi^*\mathcal{E}\otimes\mathcal{F}_{\widetilde{X},E}^p).
\end{displaymath}
For a given $p$, we have a commutative diagram of complexes
\begin{displaymath}
\small{\xymatrix{
 0\ar[r]&\Gamma(X,\mathcal{E}\otimes\mathcal{F}_{X,Y}^{p,\bullet})\ar[d]^{\pi^*} \ar[r]& \Gamma(X,\mathcal{E}\otimes\mathcal{A}_{X}^{p,\bullet})\ar[d]^{\pi^*} \ar[r]^{i_Y^*}& \Gamma(Y,i_Y^*\mathcal{E}\otimes\mathcal{A}_{Y}^{p,\bullet}) \ar[d]^{(\pi|_E)^*}\ar[r]& 0\\
 0\ar[r]&\Gamma(\widetilde{X},\pi^*\mathcal{E}\otimes\mathcal{F}_{\widetilde{X},E}^{p,\bullet})    \ar[r]^{}& \Gamma(\widetilde{X},\pi^*\mathcal{E}\otimes\mathcal{A}_{\widetilde{X}}^{p,\bullet})  \ar[r]^{i_E^*} &  \Gamma(E,i_E^*\pi^*\mathcal{E}\otimes\mathcal{A}_{E}^{p,\bullet})    \ar[r]& 0, }}
\end{displaymath}
where all the differentials are naturally induced by $\bar{\partial}$ and the two rows are exact by (\ref{relative-exact}). It induces a commutative diagram of long exact sequences
\begin{displaymath}
\tiny{\xymatrix{
    \cdots\ar[r]&H^q(X,\mathcal{E}\otimes\mathcal{F}_{X,Y}^{p}) \ar[d]^{\cong} \ar[r]& H^{p,q}(X,\mathcal{E}) \ar[d]^{\pi^*}\ar[r]^{i_Y^*}&  H^{p,q}(Y,i_Y^*\mathcal{E})\ar[d]^{(\pi|_E)^*}\ar[r]&H^{q+1}(X,\mathcal{E}\otimes\mathcal{F}_{X,Y}^{p})\ar[d]^{\cong}\ar[r]&\cdots\\
 \cdots \ar[r] & H^q(\widetilde{X},\pi^*\mathcal{E}\otimes\mathcal{F}_{\widetilde{X},E}^{p})\ar[r]& H^{p,q}(\widetilde{X},\pi^*\mathcal{E})       \ar[r]^{i_E^*}& H^{p,q}(E,i_E^*\pi^*\mathcal{E})     \ar[r] & H^{q+1}(\widetilde{X},\pi^*\mathcal{E}\otimes\mathcal{F}_{\widetilde{X},E}^{p})\ar[r]&\cdots,}}
\end{displaymath}
where $\pi^*$ and $(\pi|_E)^*$ are injective by Proposition \ref{inj-surj} and Corollary \ref{proj-bun} respectively. By the snake-lemma, $i_E^*$ induces an isomorphism
$\textrm{coker}\pi^*\tilde{\rightarrow}\textrm{coker}(\pi|_E)^*$. We  get a commutative diagram of exact sequences
\begin{equation}\label{commutative1}
\xymatrix{
 0\ar[r]&H^{p,q}(X,\mathcal{E})\ar[d]^{i_Y^*} \ar[r]^{\pi^*}& H^{p,q}(\widetilde{X},\pi^*\mathcal{E})\ar[d]^{i_E^*} \ar[r]& \textrm{coker}\pi^* \ar[d]^{\cong}\ar[r]& 0\\
 0\ar[r]&H^{p,q}(Y,i_Y^*\mathcal{E})       \ar[r]^{(\pi|_E)^*}& H^{p,q}(E,i_E^*\pi^*\mathcal{E})   \ar[r]^{} &  \textrm{coker} (\pi|_E)^*    \ar[r]& 0 },
\end{equation}
for any $p$, $q$.

\subsubsection{\emph{\textbf{Comparison of two formulas}}}
Denote
\begin{displaymath}\label{b-u-m}
\psi=\pi^*+\sum_{i=0}^{r-2}(i_E)_*\circ (h^i\cup)\circ (\pi|_E)^*:
H^{p,q}(X,\mathcal{E})\oplus \bigoplus_{i=0}^{r-2}H^{p-1-i,q-1-i}(Y,i_Y^*\mathcal{E})\rightarrow H^{p,q}(\widetilde{X},\pi^*\mathcal{E}).
\end{displaymath}
A natural question is:
\begin{quest}\label{prob.1}
\emph{Are $\phi$ and $\psi$ inverse to each other?}
\end{quest}

This question has a positive answer in following cases.
\begin{prop}\label{relation}
Suppose that
\begin{equation}\label{key2}
i_E^*i_{E*}\sigma=h\cup\sigma
\end{equation}
for any $\sigma\in H^{*,*}(E,i_E^*\pi^*\mathcal{E})$. Then $\phi$ and $\psi$ are inverse isomorphisms to each other.
\end{prop}
\begin{proof}
For any $\gamma\in H^{p,q}(\widetilde{X},\pi^*\mathcal{E})$, suppose that
\begin{displaymath}
\phi(\gamma)=(\beta^{p,q},\alpha^{p-1,q-1},...,\alpha^{p-r+1,q-r+1}),
\end{displaymath}
where $\beta^{p,q}\in H^{p,q}(X,\mathcal{E})$ and $\alpha^{p-i,q-i}\in H^{p-i,q-i}(Y,i_Y^*\mathcal{E})$ for $i=1,...,r-1$. Then $\pi_*\gamma=\beta^{p,q}$ and there exists $\alpha^{p,q}\in H^{p,q}(Y,i_Y^*\mathcal{E})$ such that
$i_E^*\gamma=\sum_{i=0}^{r-1}h^i\cup(\pi|_E)^*\alpha^{p-i,q-i}$. By (\ref{key2}),
\begin{displaymath}
i_E^*\left[\gamma-\sum_{i=0}^{r-2}(i_E)_*\left(h^i\cup(\pi|_E)^*\alpha^{p-1-i,q-1-i}\right)\right]=(\pi|_E)^*\alpha^{p,q},
\end{displaymath}
which is zero in $\textrm{coker}(\pi|_E)^*$. By the commutative diagram (\ref{commutative1}),
\begin{equation}\label{key2-1}
\gamma-\sum_{i=0}^{r-2}(i_E)_*\left(h^i\cup(\pi|_E)^*\alpha^{p-1-i,q-1-i})\right)=\pi^*\gamma^{p,q},
\end{equation}
for some $\gamma^{p,q}\in H^{p,q}(X,\mathcal{E})$. Push out (\ref{key2-1}) by $\pi_*$, $\gamma^{p,q}=\pi_*\gamma=\beta^{p,q}$ by the projection formula (\ref{pro-formula1}), where we used the fact that $(\pi|_E)_*h^i=0$ for $0\leq i\leq r-2$. Hence
\begin{displaymath}\label{b-u-m}
\gamma=\psi(\beta^{p,q},\alpha^{p-1,q-1},...,\alpha^{p-r+1,q-r+1}).
\end{displaymath}
So $\psi\phi=\textrm{id}$. By Theorem \ref{blow-up}, $\psi$ is isomorphic, and then, $\phi$ is inverse to $\psi$.
\end{proof}

\begin{prop}\label{key}
If $\mathcal{E}$ is a free sheaf of $\mathcal{O}_X$-modules of finite rank on $X$ and one of the following conditions is satisfied\emph{:}

$(1)$ $X$ and $Y$ are compact complex manifolds satisfying the $\partial\overline{\partial}$-lemma,

$(2)$ $Y$ is a Stein manifold, \\
we have $i_E^*i_{E*}\sigma=h\cup\sigma$, for $\sigma\in H^{*,*}(E,i_E^*\pi^*\mathcal{E})$. Moreover, for the two cases, $\phi$ and $\psi$ are inverse isomorphisms to each other.
\end{prop}
\begin{proof}
We just need  to prove it for the case $\mathcal{E}=\mathcal{O}_X$.

$(1)$ For a complex manifold $Z$, let $I_Z:H_{\textrm{BC}}^{p,q}(Z)\rightarrow H^{p,q}(Z)$ and  $J_Z:\bigoplus_{p+q=k}H_{\textrm{BC}}^{p,q}(Z)\rightarrow H^{k}(Z,\mathbb{C})$ be the natural maps for any $p$, $q$, $k$, where $H_{\textrm{BC}}^{*,*}(\bullet)$ denote the Bott-Chern cohomology. $X$ and $Y$ satisfy the $\partial\overline{\partial}$-lemma, so do $E=\mathbb{P}(N_{Y/X})$ and $\widetilde{X}$  (seeing \cite{ASTT}, Corollary 3  and \cite{St2}, Corollary 26). Then $I_E$, $J_E$, $I_{\widetilde{X}}$ and $J_{\widetilde{X}}$ are isomorphisms (seeing \cite{DGMS}, Remark (5.16)). We know $\mathcal{O}_E(-1)=\mathcal{O}_{\widetilde{X}}(E)|_E$,  so $[t]=c_1(\mathcal{O}_E(-1))=e(E)\in H^2(E,\mathbb{C})$, where $e(E)$ is the Euler class of $E$. By \cite{LM}, Lemma 12.2, $i_E^*i_{E*}\sigma=[t]\cup\sigma$ for any $\sigma\in H^*(\widetilde{X},\mathbb{C})$. The commutative diagrams
\begin{displaymath}
\xymatrix{
H^{p,q}_{\textrm{BC}}(E)\ar[d]^{i_{E*}} \ar[r]_{I_E}& H^{p,q}(E) \ar[d]^{i_{E*}} & \bigoplus_{p+q=k}H^{p,q}_{\textrm{BC}}(E) \ar[d]^{i_{E*}}\ar[r]_{
\qquad J_E}& H^k(E,\mathbb{C})\ar[d]^{i_{E*}}\\
H^{p+1,q+1}_{\textrm{BC}}(\widetilde{X})\ar[d]^{i_E^*} \ar[r]_{I_{\widetilde{X}}}& H^{p+1,q+1}(\widetilde{X}) \ar[d]^{i_E^*} & \bigoplus_{p+q=k+2}H^{p,q}_{\textrm{BC}}(\widetilde{X}) \ar[d]^{i_E^*}\ar[r]_{\qquad J_{\widetilde{X}}}& H^{k+2}(\widetilde{X},\mathbb{C})\ar[d]^{i_E^*}\\
H^{p+1,q+1}_{\textrm{BC}}(E)      \ar[r]_{I_E}& H^{p+1,q+1}(E)      & \bigoplus_{p+q=k+2}H^{p,q}_{\textrm{BC}}(E)     \ar[r]_{\qquad J_E}& H^{k+2}(E,\mathbb{C})}
\end{displaymath}
imply the case $(1)$.

$(2)$ It is just \cite{M2}, Lemma 4.4.
\end{proof}

Begining with $i=r-1$, inductively define polynomials $P_j^{i}(T_1,....,T_{r-2})\in \mathbb{Z}[T_1,....,T_{r-2}]$, by recursion relations
\begin{equation}\label{Polynomial-ind}
P_j^{i}(T_1,....,T_{r-2})= \left\{
 \begin{array}{ll}
(-1)^r\sum_{k=i+1}^{r-1-j}T_{k-i}P_j^{k}(T_1,....,T_{r-2}),&~j<r-1-i\\
 &\\
 (-1)^{r-1},&~j=r-1-i,
 \end{array}
 \right.
\end{equation}
where $1\leq i\leq r-1$ and $0\leq j\leq r-1-i$.
For example,
\begin{displaymath}
P_0^{r-1}=(-1)^{r-1},\mbox{ } P_1^{r-2}=(-1)^{r-1},\mbox{ } P_0^{r-2}=-T_1 \ldots.
\end{displaymath}
Suppose
\begin{displaymath}
P_j^{i}(T_1,....,T_{r-2})=\sum_{d_1,\cdots,d_{r-2}} a_{j,d_1,\cdots,d_{r-2}}^{i}T^{d_1}_1,....,T^{d_{r-2}}_{r-2}.
\end{displaymath}
By the induction for $i+j$, it is easily checked that
\begin{equation}\label{degree}
\sum_{k=1}^{r-2}kd_k=r-1-i-j,
\end{equation}
for every nonzero term  $a_{j,d_1,\cdots,d_{r-2}}^{i}T^{d_1}_1,....,T^{d_{r-2}}_{r-2}$ of $P_j^{i}(T_1,....,T_{r-2})$.

We can explicitly represent $\phi(\alpha)$ by $\alpha$ as follows.
\begin{lem}\label{represent}
For any $\alpha\in H^{p,q}(\widetilde{X},\pi^*\mathcal{E})$, if
\begin{displaymath}
\phi(\alpha)=(\pi_*\alpha,\alpha^{p-1,q-1},...,\alpha^{p-r+1,q-r+1}),
\end{displaymath}
then
\begin{equation}\label{explicit}
\alpha^{p-i,q-i}=\sum_{j=0}^{r-1-i}\left[P_j^{i}((\pi|_E)_*h^r,....,(\pi|_E)_*h^{2r-3})\right]\cup(\pi|_E)_*(h^j\cup i_E^*\alpha),
\end{equation}
for $1\leq i\leq r-1$.
\end{lem}
\begin{proof}
By the reason of degrees, $(\pi|_E)_*h^i=0$, for $0\leq i\leq r-2$.
Moreover, $(\pi|_E)_*h^{r-1}=(-1)^{r-1}$ in $H^{0,0}(Y)=\mathcal{O}(Y)$. Actually, $(\pi|_E)_*t^{r-1}$ is a $\textrm{d}$-closed smooth $0$-form on $Y$, hence a constant.  For any $y$ in $Y$,
\begin{displaymath}
\begin{aligned}
(\pi|_E)_*t^{r-1}= &\int_{E_y}t^{r-1}|_{E_y}\\
= &\int_{E_y}[\frac{i}{2\pi}\Theta(\mathcal{O}_{E_y}(-1))]^{r-1}\\
= &\int_{\mathbb{P}^{r-1}}c_1(\mathcal{O}_{\mathbb{P}^{r-1}}(-1))^{r-1}\\
= & (-1)^{r-1}.
\end{aligned}
\end{displaymath}
Push out (\ref{rep}) by $(\pi|_E)_*$, we get $\alpha^{p-r+1,q-r+1}=(-1)^{r-1}(\pi|_E)_*i_E^*\alpha$ using the projection formula (\ref{pro-formula2}). (\ref{explicit}) holds for $i=r-1$. Suppose that (\ref{explicit}) holds for any $i\geq l+1$. Cup product with (\ref{rep}) and $h^{r-1-l}$, and then push it out by $(\pi|_E)_*$, we get
\begin{displaymath}
\alpha^{p-l,q-l}=(-1)^{r-1}(\pi|_E)_*(h^{r-1-l}\cup i_E^*\alpha)+(-1)^r\sum_{k=l+1}^{r-1}(\pi|_E)_*(h^{k+r-1-l})\cup\alpha^{p-k,q-k}.
\end{displaymath}
By (\ref{Polynomial-ind}) and the inductive hypothesis, (\ref{explicit}) holds for $i=l$. We complete the proof.
\end{proof}

\begin{thm}\label{important}
$\phi$ is an isomorphism for any connected complex manifold $X$.
\end{thm}
\begin{proof}
Fix an integer $p$. By Proposition \ref{example1} $(3)$,
\begin{displaymath}
i_E^*:\mathcal{H}^{p,*}_{\widetilde{X}}(\pi^*\mathcal{E})\rightarrow i_{E*}\mathcal{H}_{E}^{p,*}(i_E^*\pi^*\mathcal{E})
\end{displaymath}
\begin{displaymath}
(\pi|_E)_*:(\pi|_E)_*\mathcal{H}_{E}^{p+j,*}((\pi|_E)^*i_Y^*\mathcal{E})[j]\rightarrow\mathcal{H}_Y^{p+j-r+1,*}(i_Y^*\mathcal{E})[j-r+1],
\end{displaymath}
and, by Proposition \ref{example1} $(2)$,
\begin{displaymath}
h^j\cup:\mathcal{H}_{E}^{p,*}(i_E^*\pi^*\mathcal{E})\rightarrow \mathcal{H}_{E}^{p+j,*}(i_E^*\pi^*\mathcal{E})[j]
\end{displaymath}
\begin{displaymath}
\prod_{k=1}^{r-2}[(\pi|_E)_*h^{k+r-1}]^{d_k}\cup:
\mathcal{H}_Y^{p+j-r+1,*}(i_Y^*\mathcal{E})[j-r+1]\rightarrow \mathcal{H}_Y^{p-i,*}(i_Y^*\mathcal{E})[-i]
\end{displaymath}
are morphisms of M-V systems of cdp presheaves, where $\sum_{k=1}^{r-2}kd_k=r-1-i-j$. Then
\begin{equation}\label{d}
i_E^*:\pi_*\mathcal{H}^{p,*}_{\widetilde{X}}(\pi^*\mathcal{E})\rightarrow \pi_*i_{E*}\mathcal{H}_{E}^{p,*}(i_E^*\pi^*\mathcal{E}),
\end{equation}
\begin{equation}\label{e}
\begin{aligned}
h^j\cup:\pi_*i_{E*}\mathcal{H}_{E}^{p,*}(i_E^*\pi^*\mathcal{E})\rightarrow &\pi_*i_{E*}\mathcal{H}_{E}^{p+j,*}(i_E^*\pi^*\mathcal{E})[j]\\
= &i_{Y*}(\pi|_E)_*\mathcal{H}_{E}^{p+j,*}((\pi|_E)^*i_Y^*\mathcal{E})[j],
\end{aligned}
\end{equation}
\begin{equation}\label{f}
(\pi|_E)_*:i_{Y*}(\pi|_E)_*\mathcal{H}_{E}^{p+j,*}((\pi|_E)^*i_Y^*\mathcal{E})[j]\rightarrow i_{Y*}\mathcal{H}_Y^{p+j-r+1,*}(i_Y^*\mathcal{E})[j-r+1],
\end{equation}
\begin{equation}\label{g}
\prod_{k=1}^{r-2}[(\pi|_E)_*h^{k+r-1}]^{d_k}\cup:i_{Y*}\mathcal{H}_Y^{p+j-r+1,*}(i_Y^*\mathcal{E})[j-r+1]\rightarrow i_{Y*}\mathcal{H}_Y^{p-i,*}(i_Y^*\mathcal{E})[-i]
\end{equation}
are M-V morphisms by Proposition \ref{elem} $(1)$, $(3)$. Combining (\ref{d})-(\ref{g}), by Section 2.2  $(5')$, $(6')$,
\begin{displaymath}
\Phi^{p-i,*-i}:=\sum_{j=0}^{r-1-i}\left[P_j^{i}((\pi|_E)_*h^r,....,(\pi|_E)_*h^{2r-3})\right]\cup(\pi|_E)_*\circ(h^j\cup)\circ i_E^*
\end{displaymath}
is a morphism $\pi_*\mathcal{H}^{p,*}_{\widetilde{X}}(\pi^*\mathcal{E})\rightarrow i_{Y*}\mathcal{H}_Y^{p-i,*}(i_Y^*\mathcal{E})[-i]$ of M-V systems of cdp presheaves on $X$ for any $1\leq i\leq r-1$,  so is
\begin{displaymath}
\Phi^*:\pi_*\mathcal{H}^{p,*}_{\widetilde{X}}(\pi^*\mathcal{E})\rightarrow \mathcal{H}^{p,*}_X(\mathcal{E})\oplus \bigoplus_{i=1}^{r-1}i_{Y*}\mathcal{H}_Y^{p-i,*}(i_Y^*\mathcal{E})[-i],
\end{displaymath}
by Section 2.2 $(2)$, $(4')$, where
\begin{displaymath}
\Phi^*=(\pi_*,\Phi^{p-1,*-1},...,\Phi^{p-r+1,*-r+1}).
\end{displaymath}

Let $\mathfrak{U}$ be a $\mathcal{E}$-free basis of $X$ such that every $U\in\mathfrak{U}$ is a Stein manifold.   Since $U_1\cap...\cap U_l\cap Y$ is Stein for  $U_1$, ..., $U_l\in\mathfrak{U}$, $\Phi^*$ is an isomorphism on $U_1\cap...\cap U_l$ by Proposition \ref{key}, i.e., $\Phi^*$ satisfies the hypothesis $(*)$. Then $\phi=\Phi^q(X)$ is an isomorphism by Theorem \ref{1.1}.
\end{proof}

\subsection{Several questions}
Let $\pi$, $X$, $Y$, $\widetilde{X}$, $E$, $t$, $\mathcal{V}$, $i_Y$, $i_E$ be defined as them of the beginning  of the present section. Denote the isomorphisms (\ref{1}) and (\ref{2}) by $\psi^{\mathcal{V}}$ and $\psi_c^{\mathcal{V}}$ respectively.  With a similar proof to Theorem \ref{important}, we get isomorphisms as (\ref{def})
\begin{displaymath}
\phi^{\mathcal{V}}:H^k(\widetilde{X}, \pi^{-1}\mathcal{V})\tilde{\rightarrow} H^k(X,\mathcal{V})\oplus \bigoplus_{i=0}^{r-2}H^{k-2-2i}(Y,i_Y^{-1}\mathcal{V})
\end{displaymath}
and
\begin{displaymath}
\phi_c^{\mathcal{V}}:H_{c}^k(\widetilde{X},\pi^{-1}\mathcal{V})\tilde{\rightarrow} H_{c}^k(X,\mathcal{V})\oplus \bigoplus_{i=0}^{r-2}H_{c}^{k-2-2i}(Y,i_Y^{-1}\mathcal{V}),
\end{displaymath}
where we use \cite{M1}, Lemma 4.3 instead of Proposition \ref{key} in the proof. As Question \ref{prob.1}, we want to ask that
\begin{quest}\label{prob.2}
$(1)$ \emph{Are $\phi^{\mathcal{V}}$ and $\psi^{\mathcal{V}}$ inverse to each other?}

$(2)$ \emph{Are $\phi_c^{\mathcal{V}}$ and $\psi_c^{\mathcal{V}}$ inverse to each other?}
\end{quest}

Analogous to Proposition \ref{relation}, it is easily to prove that
\begin{prop}\label{relation2}
$(1)$ If $i_E^*i_{E*}\sigma=h\cup\sigma$ for any $\sigma\in H^{*}(E,i_E^{-1}\pi^{-1}\mathcal{V})$, then $\phi^{\mathcal{V}}$ and $\psi^{\mathcal{V}}$ are inverse  to each other.

$(2)$ If $i_E^*i_{E*}\sigma=h\cup\sigma$ for any $\sigma\in H_c^{*}(E,i_E^{-1}\pi^{-1}\mathcal{V})$, then $\phi_c^{\mathcal{V}}$ and $\psi_c^{\mathcal{V}}$ are inverse to each other.
\end{prop}

For Question \ref{prob.2},  we are interested in the conditions in Proposition \ref{relation2}. Actually, we may consider the more general cases:
\begin{quest}\label{prob.3}
\emph{Let $Y$ be an oriented connected submanifold of an oriented connected smooth manifold $X$ and $i:Y\rightarrow X$  the inclusion. Assume that $\mathcal{V}$ is a local system of $\mathbb{R}$-modules of finite rank on $X$ and $[Y]$ is the fundamental class of $Y$ in $X$. Do we have the following formulas:}

$(1)$ $i^*i_*\sigma=[Y]|_Y\cup\sigma$\emph{, for $\sigma\in H^*(Y,i^{-1}\mathcal{V})$?}

$(2)$ $i^*i_*\sigma=[Y]|_Y\cup\sigma$\emph{, for $\sigma\in H_{c}^*(Y,i^{-1}\mathcal{V})$?}
\end{quest}
This question has a positive answer for the weight $\theta$-sheaf $\mathcal{V}=\underline{\mathbb{R}}_\theta$, seeing \cite{M1}, Lemma 4.3. An analogue of Theorem \ref{blow-up} on Bott-Chern cohomology is as follows.
\begin{quest}\label{prob.4}
\emph{Assume that $X$, $Y$, $\pi$, $\widetilde{X}$, $E$, $t$,  $i_E$, $r$ are defined as them of the beginning of the present section. Set $h_{\textrm{BC}}=[t]_{\textrm{BC}}\in H_{\textrm{BC}}^{1,1}(E)$. For any $p$, $q$, is}
\begin{equation}\label{BC-blow-up}
\pi^*+\sum_{i=0}^{r-2}(i_E)_*\circ (h_{\emph{BC}}^i\cup)\circ (\pi|_E)^*:H_{\emph{BC}}^{p,q}(X)\oplus \bigoplus_{i=0}^{r-2}H_{\emph{BC}}^{p-1-i,q-1-i}(Y)\rightarrow H_{\emph{BC}}^{p,q}(\widetilde{X})
\end{equation}
\emph{an isomorphism?}
\end{quest}
If compact complex manifolds $X$ and $Y$ both satisfy the $\partial\bar{\partial}$-lemma, then there is a canonical isomorphism between their Bott-Chern and Dolbeault cohomologies. So the question has a positive answer  by Theorem \ref{blow-up}. On general compact complex manifolds, Yang, S., Yang, X.-D. (\cite{YY}, Thm. 1.2) and Stelzig, J. (\cite{St1}, Prop. 4, Thm. 8 and \cite{St2}, Cor. 12) proved that there exists an isomorphism
\begin{displaymath}
H_{\textrm{BC}}^{p,q}(X)\oplus \bigoplus_{i=0}^{r-2}H_{\textrm{BC}}^{p-1-i,q-1-i}(Y)\cong H_{\textrm{BC}}^{p,q}(\widetilde{X}).
\end{displaymath}
We don't know if this isomorphism can be given by (\ref{BC-blow-up}).
For the noncompact cases, a possible approach is similar to that used in the proof of Theorem \ref{blow-up}, where $i_E^*i_{E*}\sigma=h\cup\sigma$, for $\sigma\in H_{\textrm{BC}}^{*,*}(E)$ and the Mayer-Vietoris sequence for Bott-Chern cohomology are necessary. But we don't know if they hold.


\end{document}